\documentclass[12pt]{article} 
\pdfoutput=1
\usepackage{amsmath,amsthm,amssymb,bbm,color,enumerate}
\usepackage[noend]{algorithmic}

\usepackage{graphicx}
\usepackage{epstopdf}
\newcommand\figscale{0.45}

\usepackage[UKenglish]{babel}
\usepackage{lmodern}
\usepackage{microtype}
\usepackage{mathtools}
\usepackage{esint}
\usepackage{comment}
\usepackage{hyperref}
\usepackage{authblk}
\usepackage[utf8]{inputenc}

\mathtoolsset{showonlyrefs}   

\theoremstyle{plain}
\newtheorem{theorem}{Theorem}

\newtheorem{lemma}[theorem]{Lemma}
\newtheorem{corollary}[theorem]{Corollary}

\newtheorem{proposition}[theorem]{Proposition}

\theoremstyle{definition}
\newtheorem{algorithm}{Algorithm}

\theoremstyle{remark}
\newtheorem{remark}[theorem]{Remark}

\newcommand*{\R}{\mathbb{R}}

\newcommand*{\Z}{\mathbb{Z}}

\newcommand*{\eps}{\varepsilon}
\newcommand*{\doo}{\partial}
\newcommand*{\ol}[1]{\overline{#1}}
\newcommand{\sulut}[1]{\left( #1 \right)}
\newcommand{\joukko}[1]{\left\{ #1 \right\}}
\newcommand{\abs}[1]{\left\lvert #1 \right\rvert}

\newcommand{\norm}[1]{\left\| #1 \right\|}
\newcommand{\der}{\mathrm{d}}

\DeclareMathOperator{\dive}{div}

\DeclareMathOperator{\real}{Re}

\title{Boundary determination for hybrid imaging from a single measurement}

\author{Tommi Brander \\ tommi.brander@ntnu.no}
\affil{Norwegian University of Science and Technology, Department of Mathematical Sciences}
\affil{Technical University of Denmark, Department of Applied Mathematics and Computer Science}

\author{Torbjørn Ringholm \\ 	ringholm@gmail.com}
\affil{Norwegian University of Science and Technology, Department of Mathematical Sciences}

\begin{document}

\maketitle

\begin{abstract}
We recover the conductivity~$\sigma$ at the boundary of a domain from a combination of interior and boundary data, with a single quite arbitrary measurement, in AET or CDII.
The argument is elementary and local.
More generally, we consider the variable exponent $p(\cdot)$-Laplacian as a forward model with the  interior data $\sigma |\nabla u|^q$,  and find out that single measurement specifies the boundary conductivity when $p- q \ge 1$, and otherwise the measurement specifies two alternatives.
We present heuristics for selecting between these alternatives.
Both $p$ and $q$ may depend on the spatial variable $x$, but they are assumed to be a priori known.
We illustrate the practical situations with numerical examples.
\end{abstract}


%






\section{Introduction}

Calder\'on's problem~\cite{Calderon:1980} asks if the electric conductivity $\sigma$ in an object~$\Omega$ can be reconstructed from boundary measurements of current and voltage given by the Dirichlet-to-Neumann map (DN~map) $u|_{\doo \Omega} \to \sigma \nabla u \cdot \nu |_{\doo \Omega}$, where $\nu$ is the unit outer normal.
Positive results often use an infinite number of measurements (but not always, see e.g.~\cite{Kang:Seo:2000,Kar:Wang}), even for boundary determination where one only wants to know the conductivity at the boundary~$\doo \Omega$.
Instead of an infinite number of measurements, one can think of many methods as using boundary values with infinitely large and focused oscillations.

Hybrid or coupled physics imaging methods combine multiple physical processes to acquire interior data about the object and use the interior data to determine its physical properties.
This article is concerned with acousto-electric tomography (AET), which is also known as ultrasound mediated electrical impedance tomography (UMEIT), and with current density impedance imaging (CDII).
These imaging modalities provide pointwise interior data of the form $H = \sigma \abs{\nabla u}^q$ for the exponents $q = 1$ (CDII) and $q = 2$ (AET/UMEIT); also, magnetic resonance electrical impedance tomography (MREIT) has been researched as giving the information corresponding to $q=1$. \cite{Bal:2013,Kuchment:Steinhauer:2012,Kwon:Woo:Yoon:Seo:2002}

Our main observation is that any measurement of Dirichlet data~$D$, Neumann data~$N$ and interior data~$H$  (corresponding to a non-constant, bounded solution~$u$) already contains a fair amount of information on the conductivity on the boundary.
We provide an algorithm for boundary reconstruction in dimension two from any such combination of data for AET and CDII.
The main theorem is theorem~\ref{thm:boundary_determination}, which gives conditions on $\nabla u$ and other parameters at a given boundary point~$x_0$ that determine whether there exists a unique conductivity $\sigma(x_0)$ that can be recovered, whether there exist two conductivity candidates $\sigma_\pm$ among which we have to select the correct one (some conditions for this are provided in section~\ref{sec:select}), or whether $\sigma(x_0)$ remains completely unknown at the particular point.
Our theoretical results cover a wider range of non-linear equations, also with a variable exponent in the interior data.
We give a broader explanation of the main theorem and background for boundary determination in the following subsection~\ref{sec:intro_determination}.

In corollaries~\ref{cor:p2q1} and \ref{cor:pq_2} we give explicit reconstruction formulas for the conductivity in the cases of AET/UMEIT and CDII.
In section~\ref{sec:numeric} we provide reconstruction algorithms and use simulated data for reconstructions.

\subsection{Boundary determination with interior data}
\label{sec:intro_determination}
Alberti and Capdeboscq provide an introduction to the mathematics of hybrid data imaging in~\cite{Alberti:Capdeboscq:2018}.
With interior data of power density type it is possible to reconstruct the conductivity with a small number of suitably chosen measurements~\cite{Bal:2013}.
In the present article we investigate what we can say from a single measurement with arbitrary boundary values.
It turns out that it is possible to recover the conductivity uniquely or almost uniquely at the boundary of the domain, assuming sufficient regularity from the conductivity and the boundary.
Our argument is elementary, local and similar to a boundary determination argument for the $p$-Laplacian~\cite[lemma 4.2]{Brander:2016:jan}.

The question of boundary determination is relevant since some results for hybrid inverse problems assume the conductivity is known on or close to the boundary~\cite{Bal:2013:aug,Capdeboscq:Fehrenbach:deGournay:Kavian:2009,Nachman:Tamasan:Timonov:2011}.
Boundary determination could also be used to calibrate a measurement device or measure errors in the devices.
Our boundary determination algorithm can use data of arbitrary measurements and is not computationally demanding, which suggests that it can be added to any other reconstruction method as a verification step or to improve the reconstruction at boundary.

Our method works generally for power densities $\sigma \abs{\nabla u}^{q(x)}$ with arbitrary and varying power $q(x) \ge 0$, though the case $q = 0$ is trivial.
We omit the physics of the hybrid data imaging from the present paper and instead refer to the book of Alberti and Capdeboscq~\cite[section 1.2]{Alberti:Capdeboscq:2018}.
We note that only the powers $q \equiv 1$ and $q \equiv 2$ are relevant for presently known applications and they come from very different physical processes, so even interpolation or variation of the parameters is not feasible in an obvious manner.
Furthermore, our method works when the forward model is the non-Ohmic $p(\cdot)$-conductivity equation, where $1 < p^- < p(x) < p^+ < \infty$ and
\begin{equation}
\dive \sulut{\sigma(x) \abs{\nabla u}^{p(x)-2}\nabla u} = 0,
\end{equation}
where the case $p \equiv 2$ is the usual linear and Ohmic conductivity equation
\begin{equation}
\dive \sulut{\sigma(x) \nabla u} = 0.
\label{eq:linear_cond_eqn}
\end{equation}
Physically, the conductivity equation follows from Ohm's law
\begin{equation}
I = \sigma \nabla u,
\end{equation}
where $I$ is the electric current and $u$ is the electric potential (voltage),
and from Kirchhoff's law
\begin{equation}
\dive I = 0.
\end{equation}
Since Ohm's law is an approximation based on empirical data, and the current-voltage characteristic is in general a complicated non-linear one (and might not even be a function), it is of interest to consider more general non-linear Ohm's laws.
In the present work we consider a power-law type Ohm's law, where the type of the power law relation may vary spatially; namely,
\begin{equation}
I = \sigma \abs{\nabla u}^{p(x)-2} \nabla u.
\end{equation}
This leads to the variable exponent $p(\cdot)$-Laplace equation.
An example of a power-law type Ohm's law is certain polycrystalline compounds near the transition to superconductivity~\cite{Bueno:Longo:Varela:2008,Dubson:Herbert:Calabrese:Harris:Patton:Garland:1988}, where the exponent~$p$ is a function of temperature.

Calder\'on's problem for the nonlinear model with constant $p$ was introduced by Salo and Zhong~\cite{Salo:Zhong:2012}; for a review, see the thesis~\cite{Brander:2016:apr}.
The known boundary determination results for the $p$-Laplacian use an arbitrarily large parameter which causes the solutions to oscillate~\cite{Brander:2016:jan,Brander:Harrach:Kar:Salo:2018,Salo:Zhong:2012}.
The variable exponent equation has been investigated in one dimension~\cite{Brander:Winterrose}, where non-injectivity of the exponent~$p$ provides the only obstacle to recovering the conductivity from the DN~map.
In one dimension, in addition to the previous result, interior data is sufficient to solve Calderón's problem at all points~$x$ where $p(x) - q(x) \neq 1$~\cite[remark 10]{Brander:Winterrose}, a condition that also plays a role in this paper.

Investigating the inverse problem with the parameters $p$ and $q$ reveals curious properties; see section~\ref{sec:inverse} for proofs.
\begin{enumerate}
\item
At boundary points $x$ where $p(x) - q(x) > 1$, the conductivity $\sigma(x)$ can be recovered if $\nabla u (x) \neq 0$.
\item
Where $p(x) - q(x) = 1$, the conductivity $\sigma(x)$ can be recovered if the component of $\nabla u(x)$ that is tangent to the boundary does not vanish.
When it does vanish, the interior data and the Neumann data are equal and nothing can be deduced about the conductivity.
\item
Where $p(x) - q(x) < 1$, two candidates for the conductivity can be recovered at all points where both the tangential and the normal components of $\nabla u(x)$ are nonzero. If $\nabla u(x) \neq 0$ and either its tangential or normal component vanishes, then the conductivity can be recovered uniquely.
The two candidates are equal if and only if the absolute values of $\nabla u(x) \cdot \nu(x)$ and the component of $\nabla u(x)$ tangent to the boundary have a specific relationship, which depends on the value of~$(p-q)(x)$.
We present some situations where the correct candidate can be selected in section~\ref{sec:select}.
\end{enumerate}
In fact, if we consider the problem as a $p(\cdot)$-Laplace equation with the interior data~$H = \sigma \abs{\nabla u}^{q(x)}$ as a weight function, we notice that
\begin{equation}
\dive\sulut{\sigma(x) \abs{\nabla u}^{p(x)-2} \nabla u}
= \dive\sulut{H(x) \abs{\nabla u}^{p(x)-q(x)-2} \nabla u}.
\end{equation}
Such an equation is elliptic if $p(x)-q(x) > C > 1$ everywhere and hyperbolic if $p(x)-q(x) < c < 1$ everywhere, with the case $p(x) -q(x) = 1$ being degenerate elliptic~\cite[section 4]{Bal:2013:aug}\cite[theorem 3.2]{Bal:Hoffman:Knudsen:2017}.

\subsection*{Acknowledgements}
T.B.\ was partially funded by grant no.\ 4002-00123 from the Danish Council for Independent Research | Natural Sciences, and partially by the Research Council of Norway through the FRIPRO Toppforsk project ''Waves and nonlinear phenomena''.
We would like to thank Changyou Guo for discussions and early numerical results.

\section{Forward problem}
\label{sec:forward}
Let $1 < p(x) < \infty$ and suppose $p \colon \Omega \to \R$ is a measurable function, with $\Omega \subset \R^d$ with $d \geq 2$.
(For $d=1$ we refer to the work of Brander and Winterrose~\cite{Brander:Winterrose}.)
We first discuss the existence and uniqueness of the weighted variable exponent equation and after that state a regularity result.

Before proceeding, we define the variable exponent Lebesgue space $L^p(\Omega$), with $\Omega \subset \R^d$ a bounded open set and $d \geq 1$.
The variable exponent Sobolev spaces are defined in terms of $L^p(\Omega)$ in the usual way.
Following the book of Diening, Harjulehto, Hästö and R\r{u}\v{z}i\v{c}ka~\cite[sections 2 and 3]{Diening:Harjulehto:Hasto:Ruzicka:2011},
\begin{equation}
    L^p(\Omega) = \joukko{f \colon \Omega \to \R \text{ measurable } ; \lim_{\lambda \to 0} \int_\Omega \abs{\lambda f(x)}^{p(x)} \der x = 0 },
\end{equation}
where functions which agree almost everywhere are considered identical, and
\begin{equation}
\norm{f}_{L^p(\Omega)} = \inf \joukko{\lambda > 0 ; \int_\Omega \abs{\frac{f(x)}{\lambda}}^{p(x)} \der x \le 1}.
\end{equation}
These correspond to the classical Lebesgue spaces and norms if $p$ is constant~\cite[example~2.1.8]{Diening:Harjulehto:Hasto:Ruzicka:2011}.

The Dirichlet problem for the varying exponent $p(\cdot)$-Laplacian is
\begin{align}
\dive \sulut{ \sigma \abs{\nabla u}^{p(x)-2} \nabla u } = 0 &\text{ in } \Omega \\
u = f &\text{ on } \doo \Omega.
\end{align}
The equation is the Euler-Lagrange equation of the energy
\begin{equation}
\label{eq:normal_energy}
v \mapsto \int_\Omega \frac{\sigma}{p(x)} \abs{\nabla v}^{p(x)} \der x.
\end{equation}
We assume that the Dirichlet boundary values are bounded.
If this is not the case, the situation becomes more complicated~\cite[section 13]{Diening:Harjulehto:Hasto:Ruzicka:2011} and it is necessary to impose additional properties on $p$ and the domain.
\begin{lemma}
Suppose $1 < p^- \leq p(x) \leq p^+ < \infty$, and that $\Omega \subset \R^d$, $d \in \Z_+$, is a bounded open set that supports the Poincaré inequality with $p \equiv 1$.
Consider boundary values $f \in W^{1,p(\cdot)} \cap L^\infty(\Omega)$.
Then there exists a unique minimizer in $W^{1,p(\cdot)} \cap L^\infty(\Omega) + f$ of the energy~\eqref{eq:normal_energy}.
\end{lemma}
Recall that the 1-Poincaré inequality is satisfied for example in John domains~\cite[section 8.2]{Diening:Harjulehto:Hasto:Ruzicka:2011}, and in particular in Lipschitz domains.
\begin{proof}
The proof uses the direct method in the calculus of variations.
The variable exponent Sobolev space is a reflexive Banach space~\cite[theorem 8.1.6]{Diening:Harjulehto:Hasto:Ruzicka:2011} and the functional is convex, since $t \mapsto c t^p$ is convex for all $p \geq 1$ and $c \geq 0$.
The energies are lower semicontinuous~\cite[theorem 3.2.9 and section 3.2]{Diening:Harjulehto:Hasto:Ruzicka:2011}.
Coercivity of the functional requires the Poincaré inequality with $p\equiv 1$ (since we assume bounded boundary values).
Therefore the functional has a unique minimizer.
\end{proof}

Suppose $\Omega$ is a bounded open set that is smooth enough for $\nabla u \in C\sulut{\ol \Omega}$, and 
also suppose the conductivity and the boundary values are smooth enough; see the regularity lemma, lemma~\ref{lemma:p_regularity}, for sufficient conditions.
Then the voltage-to-current, or Dirichlet-to-Neumann, map is
\begin{equation}
\Lambda_\sigma (u) = \sigma \abs{\nabla u}^{p(x)-2} \nabla u \cdot \nu
\end{equation}
in its strong form.
One typically defines the map in the weak sense~\cite[section 2]{Brander:Winterrose}, but we make no use of the weak definition in the present work.

In the general nonlinear setting the following lemma gives sufficient conditions for the boundary regularity:
\begin{lemma}[Regularity] \label{lemma:p_regularity}
Let $0 < \beta \leq 1$ and $1 < p^- \leq p(x) \leq p^+ < \infty$.
Suppose $\Omega$ is a bounded open $C^{1,\beta}$ set,
the exponent $p(\cdot)$ is Hölder continuous in $\ol \Omega$,
and suppose the conductivity $0 < \sigma \in C^{0,\beta}(\ol{\Omega})$ is bounded from above.
Consider the weighted $p$-Laplace equation with Dirichlet boundary values $f \in C^{1,\beta}(\doo \Omega)$ or Neumann boundary values ~$N \in C^{1,\beta}(\doo \Omega)$.
Then the solution~$u$ of the weighted $p(\cdot)$-Laplace equation is in $C^{1,\gamma}(\ol{\Omega})$ for some $\gamma > 0$.
\end{lemma}
A proof of the lemma can be found in a paper of Fan~\cite[theorems 1.2. and 1.3]{Fan:2007}.

\section{Boundary determination}
\label{sec:inverse}

In this section we always assume that $\nabla u \in C\sulut{\ol \Omega}$ (see lemma~\ref{lemma:p_regularity} for sufficient conditions for this) and $\doo \Omega$ is $C^1$-smooth.

In AET and CDII we have several different kinds of measurement data.
\begin{description}
\item[Dirichlet data] $D = u|_{\doo \Omega}$ is the boundary potential, i.e.\ electric voltage on the boundary.
\item[Neumann data] $N = \sigma \abs{\nabla u}^{p(x)-2} \nabla u \cdot \nu$ is the current flux out of the domain.
\item[Interior data] $H = \sigma \abs{\nabla u}^{q(x)}$ with $0 \le q(x) < \infty$ is, if $q = 2$, the electric power density, and if $q = 1$, the current flux density.
\end{description}

The Dirichlet data also lets us calculate the component of $\nabla u$ tangent to the boundary at boundary points.
Suppose that at every boundary point $x \in \doo \Omega$ the vectors $\nu, \alpha_1,\ldots,\alpha_j,\ldots,\alpha_{d-1}$ are orthonormal.
Then, supposing the boundary of the domain is $C^1$, we can calculate for every $1 \leq j \leq d-1 $ the quantity $\nabla u \cdot \alpha_j $ from the Dirichlet data.
We fix a boundary point $x \in \doo \Omega$, omit it from the notation, and write
\begin{align}
A &= \sqrt{\sum_{j = 1}^{d-1} \abs{\nabla u \cdot \alpha_j}^2} \\
n &= \nabla u \cdot \nu.
\end{align}
Note that $A$ is calculated from the Dirichlet data~$D$, which is known, while $n$ is unknown.
Also note that $A$ is independent of the choice of the vectors~$\alpha_j$.

\subsection{Boundary determination at a point}
\label{sec:bdry_at_point}

If $\nabla u = 0$, then $A = N = H = 0$ and we can recover nothing.

If $\nabla u \cdot \nu = 0 $ but $\nabla u \neq 0$, then we can recover conductivity from the interior data and Dirichlet data:
\begin{equation}
\sigma = H \abs{\nabla u}^{-q} = H A^{-q}.
\end{equation}

If $A= 0$ and $N \neq 0$, then we can reconstruct~$\sigma$ when $\abs{N} \neq H$, which happens if and only if $p-1 \neq q$.

We now consider the general case, where $N \neq 0$, $A \neq 0$ and $H \neq 0$.
We want to solve the nonlinear pair of equations
\begin{align}
N &= \sigma \sulut{A^2 + n^2}^{(p-2)/2}n \label{eq:N}
\\
H &= \sigma \sulut{A^2+n^2}^{q/2}, \label{eq:H}
\end{align}
where $N$, $A$, $H$, $p$ and $q$ are known quantities, and $n$ and $\sigma$ are the unknowns.
Dividing equation~\eqref{eq:N} by equation~\eqref{eq:H} we get
\begin{align}
\label{eq:g_n}
N / H = \sulut{A^2 + n^2}^{\sulut{p - q - 2}/2}n.
\end{align}
Every solution~$n$ to equation~$\eqref{eq:g_n}$ also gives a possible solution~$\sigma$ to the pair of equations~\eqref{eq:N} -- \eqref{eq:H}.
We write 
\begin{equation}
g(n) = \sulut{A^2 + n^2}^{\sulut{p - q - 2}/2}n.
\end{equation}
Without loss of generality we may assume $n > 0$, since $n$ and $N$ have the same sign and the other variables (except the power~$p-q$) are positive.

We first observe that for all values of $p-q$ we have
\begin{equation}
\lim_{n \to 0} g(n) = 0.
\end{equation}
When $p-q > 1$, we also have
\begin{equation}
\lim_{n \to \infty} g(n) = \infty.
\end{equation}
When $p-q = 1$, we instead have
\begin{equation}
\lim_{n \to \infty} g(n) = 1,
\end{equation}
and when $p-q < 1$, we get
\begin{equation}
\lim_{n \to \infty} g(n) = 0.
\end{equation}

Since $g'(n) > 0$ when $p-q \geq 1$, $g$ is strictly increasing.
For $p-q < 1$, $g$ increases strictly until
\begin{equation}
-(p-q-1)n^2 = A^2,
\end{equation}
and decreases strictly after that.

Thence: For $p-q < 1$, we may have one or two solutions to equation~\eqref{eq:g_n}, and for $p-q \ge 1$, we have exactly one potential solution.
We can then solve for $\sigma$ from $H$.
Substituting this into the formula for $N$, equation~\eqref{eq:N} verifies that all the potential solutions do indeed solve the pair of equations.
We have thus proved the following theorem.
\begin{theorem}
\label{thm:boundary_determination}
Let $\Omega \subset \R^d$ be an open set, and suppose that $\Omega, f, p$ and $\sigma$ are such that the weighted $p(\cdot)$-Laplace equation has a unique solution~$u \in C^1\sulut{\ol{\Omega}}$.

Then, from the combined Dirichlet data, Neumann data and interior data~$H(x) = \sigma \abs{\nabla u}^{q(x)}$ we can recover the following at a boundary point~$x$:
\begin{itemize}
\item If $H(x) = 0$, nothing.
\item If $N(x) = 0 $ but $A(x) \neq 0$, then we can recover conductivity from the interior and Dirichlet data:
\begin{equation}
\sigma(x) = H(x) \abs{\nabla u(x)}^{-q(x)} = H(x) \sulut{A(x)}^{-q(x)}.
\end{equation}
\item If $A(x)= 0$ and $N(x) \neq 0$, then we can reconstruct~$\sigma(x)$ if and only if $p(x)-q(x) \neq 1$:
\begin{equation}
\begin{split}
\sigma &= H^{1+q/(p-q-1)} N^{-q/(p-q-1)} \\
&= N^{1-(p-1)/(p-q-1)} H^{(p-1)/(p-q-1)}.
\end{split}
\end{equation}
\item If $A(x) \neq 0$ and $N(x) \neq 0$, then the pair of equations
\begin{align}
N &= \sigma \sulut{A^2 + n^2}^{(p-2)/2}n \\
H &= \sigma \sulut{A^2+n^2}^{q/2},
\end{align}
has two pairs of solutions $(\sigma,n)$, both of which yield candidates for the conductivity when $p-q < 1$ and $-\sulut{p-q-1}n^2 \neq A^2$.
If $p-q \ge 1$, or $p-q < 1$ and $-\sulut{p-q-1}n^2 = A^2$, the equations have only one solution pair.
\end{itemize}
\end{theorem}

For convenience, we state the formulas for conductivity that are relevant for the currently researched medical imaging modalities of interest.
They can be recovered by explicitly solving the equations, which is possible for values of $p-q$ that turn equation~\eqref{eq:g_n} into a polynomial equation of small order.
\begin{corollary}
\label{cor:p2q1}
\begin{description}
\item[CDII/MREIT, $p - q \equiv 1$.] If $A=0$, we can say nothing. If  $A \neq 0$, we have a unique conductivity that agrees with the measurements:
\begin{equation}
\sigma = \frac{\sqrt{H^2-N^2}}{A}.
\end{equation}
\end{description}
\end{corollary}
\begin{corollary}
\label{cor:pq_2}
\begin{description}
\item[UMEIT/AET, $p - q \equiv 0$.] Since $p - q < 1$, we expect two candidate values of the normal derivative~$n$ and thereby two conductivity candidates that agree with the measurements:
\begin{align}
n_\pm &=  \frac{H}{2N}\sulut{1 \pm \sqrt{1 - 4 A^2N^2/H^2}}
\label{eq:n_pm}
\\
\sigma_\pm &= \frac{N}{n_\pm} = \frac{H}{n_\pm^2 + A^2} =  \frac{2N^2}{H\sulut{1\pm\sqrt{1-4A^2N^2/H^2}}}.
\label{eq:sigma_pm}
\end{align}
This formula is true when $N \neq 0$ and $A \neq 0$. Otherwise:
\begin{itemize}
\item If $H = 0$, then 
we can say nothing.
\item If $N = 0$ but $A \neq 0$, then $\sigma = H A^{-2}$.
\item If $A = 0$ but $N \neq 0$, then $\sigma = N^2 H^{-1}$.
\end{itemize}

\end{description}
\end{corollary}

The next lemma states that the conductivity candidates are ordered.
It is used in algorithm~\ref{alg:hard_alg} when checking the equality of the candidates.
\begin{lemma}
When $p=q$, we have
$\abs{n_-} \le \abs{n_+}$
and
$\sigma_+ \le \sigma \le \sigma_-$.
\label{lem:n_difference}
\end{lemma}
\begin{proof}
Since $\sigma(x)$ equals one of the $\sigma_\pm(x)$, it is sufficient to prove $\sigma_+(x) \le \sigma_-(x)$ at every boundary point $x \in \doo \Omega$.
From~\eqref{eq:n_pm} we have $\abs{n_+} \ge \abs{n_-}$, which by~\eqref{eq:sigma_pm} implies $\sigma_+ \le \sigma_-$.
\end{proof}

\subsection{Selecting the right candidate}
\label{sec:select}

The following propositions allow the unique recovery of conductivity in the case $p-q < 1$ around points where one of the candidate conductivities goes to infinity or zero.
Algorithm~2 in section~\ref{sec:algorithm} uses a similar idea when it checks whether both candidate conductivities are within the a priori bounds. 
The results below state that the bounds will not be satisfied at certain points.
\begin{proposition}
Consider an open set $\Omega \subset \R^d$ and a boundary point $x_0 \in \doo \Omega$ such that:
\begin{itemize}
\item $A(x_0) > 0$
\item $ - \infty < c < p(x)-q(x) < C < 1$ in a neighbourhood of $x_0$
\item $N(x_0) = 0$.
\end{itemize}
Then we can uniquely determine $\sigma(x)$ in a neighbourhood of $x_0$, and the false candidate for conductivity has the limit zero at $x_0$.
\end{proposition}
\begin{proof}
By theorem~\ref{thm:boundary_determination}, $\sigma(x_0)$ is uniquely determined and there are at most two candidates in its neighbourhood.
By considering a sufficiently small neighbourhood, both of the following are true therein:
\begin{itemize}
\item $A(x_0)/2 \le A(x) \le2 A(x_0)$
\item $ - \infty < c < p(x)-q(x) < C < 1$.
\end{itemize}
In the proof we sometimes omit the variable~$x$ from estimates for the sake of readability.
There is no~$x_0$ in the estimates.

Since $N(x) \to 0$ as $x \to x_0$, we consider the inequality
\begin{equation}
\eps > \abs{N(x)}/H(x) = \sulut{A^2 + n^2}^{\sulut{p - q - 2}/2}\abs{n},
\end{equation}
where $\eps > 0$,
and try to solve the candidates for $n$ and thereby the candidates for $\sigma$ based on the available information.
We cannot have $A^2(x) = n^2(x) $ infinitely close to $x_0$, as this would imply $N(x) \not \to 0$.

If $\abs{n(x)} < A(x)$, then we estimate
\begin{equation}
\eps > \sulut{A^2 + n^2}^{\sulut{p - q - 2}/2}\abs{n} > \sulut{\sqrt{2} A}^{p-q-2} \abs{n},
\end{equation}
which only goes to zero if $n(x) \to 0$, since the $\sulut{\sqrt{2} A}^{p-q-2}$ term is bounded.
If $\abs{n(x)} > A(x)$, then we estimate
\begin{equation}
\begin{split}
\eps &> \sulut{A^2 + n^2}^{\sulut{p - q - 2}/2}\abs{n} >  \sqrt{2}^{p-q-2} \abs{n(x)}^{p-q-1},
\end{split}
\end{equation}
which only goes to zero when $\abs{n(x)} \to \infty$, since $p-q < 1$.

We have thus deduced that either $n(x) \to 0$ or $\abs{n(x)} \to \infty$.
Since the solution $u \in C^1\sulut{\ol{\Omega}}$, the normal derivative~$n = \nabla u \cdot \nu$ must be bounded and $n \to 0$.
By continuity, this identifies the correct value of $n$ and hence also $\sigma$.

In particular, if we had $\abs{n(x)} \to \infty$, then due to the equation
\begin{equation}
\sigma(x) \abs{A(x)^2+n(x)^2}^{(p(x)-2)/2}n(x) = N(x) \to 0
\end{equation}
and boundedness of $A$, we would have $\sigma(x) \to 0$.
\end{proof}

\begin{proposition}
Consider an open set $\Omega \subset \R^d$ and a boundary point $x_0 \in \doo \Omega$ such that:
\begin{itemize}
\item $A(x_0) = 0$
\item $ - \infty < c < p(x)-q(x) < C < 1$ in a neighbourhood of $x_0$
\item $N(x_0) \neq 0$
\item $0 < c < \abs{N(x)}/H(x) < C < \infty$ in a neighbourhood of $x_0$.
\end{itemize}
Then we can uniquely determine $\sigma(x)$ in a neighbourhood of $x_0$, and the false candidate is not bounded in this neighbourhood.
\end{proposition}

\begin{proof}
By theorem~\ref{thm:boundary_determination}, $\sigma(x_0)$ is uniquely determined and there are at most two candidate pairs~$(n,\sigma)$ in a neighbourhood of $x_0$.
For convenience we assume $N, n \ge 0$.
We have
\begin{equation}
n(x_0) = \sulut{\frac{N(x_0)}{H(x_0)}}^{1/\sulut{p(x_0)-q(x_0)-1}} \neq 0.
\end{equation}
Due to continuity of $n$ and $\sigma$, we know that there is a candidate pair that converges to $\sulut{n(x_0), \sigma(x_0)}$ as $x \to x_0$.
If we can demonstrate that the other pair does not converge to the same values, then the lemma is proven.

We search for the false candidate pair $\sulut{\bar n, \bar \sigma}$ such that, as $x \to x_0$, we get $\sulut{\bar \sigma(x) , \bar n(x)} \to \sulut{\infty, 0}$, and $\bar n(x) / \bar A (x) \to 0$.
Leaving $x$ implicit, we have
\begin{equation}
\begin{split}
\frac{N}{H} &= n\sulut{A^2+n^2}^{(p-q-2)/2} \\
&= A^{p-q-1}\frac{n}{A}\sulut{1+\sulut{\frac{n}{A}}^2}^{(p-q-2)/2}.
\end{split}
\end{equation}
This must be bounded due to the boundedness assumption on $N/H$. Since $A(x) \to 0$ as $x \to x_0$ and $p(x)-q(x) < C <1$, we must have either $n(x)/A(x) \to 0$ or $n(x)/A(x) \to \infty$.

If $n(x)/A(x) \to 0$ as $x \to x_0$, then in particular $n(x) \to 0$, and  due to 
\begin{equation}
H = \sigma\sulut{A^2+n^2}^{q/2},
\end{equation}
we must have $\sigma(x) \to \infty$ as $x \to x_0$.
Thus, the condition $n(x)/A(x) \rightarrow 0$ results in the false candidate pair $(\bar{n},\bar{\sigma})$. 
\end{proof}

\begin{remark}
In general there is no hope of recovering the value of conductivity on the entire boundary based on the results of this paper only. One can construct an example where, on a part $\Gamma$ of the boundary, there are several points where $n^2 = A^2$ with non-vanishing $n$ and $A$ between. 
This will lead to two conductivity candidates on almost all of $\Gamma$.

Indeed, consider a Cauchy problem on a flat, open subset~$\Gamma$ of the boundary.
Suppose $p = q \equiv 2$.
We can then select the Dirichlet data so that $A \equiv 1$ and $\sigma(x) \nabla u(x) \cdot \nu = N(x) \equiv 1$ on~$\Gamma$.
Suppose the conductivity is analytic and oscillates around the value 1 on $\Gamma$, which means that $n(x) = N(x)/\sigma(x)$ also oscillates and is analytic.

By Cauchy-Kovalevski theorem~\cite[chapter 3]{John:1982} the Cauchy problem for the conductivity equation has a (possibly non-unique) solution~$u$ in a neighbourhood of~$\Gamma$.
We consider a domain~$\Omega$ contained in this neighbourhood and take $u|_{\doo \Omega}$ as Dirichlet values.
Then this problem has a unique solution with the desired boundary behaviour.

\end{remark}

\section{Algorithms and numerical experiments}
\label{sec:numeric}
\subsection{Algorithms}
\label{sec:algorithm}

The following two algorithms implement the boundary reconstruction result in two dimensions, $d=2$, and for the linear equation,  $p \equiv 2$.
The dimension simplifies the algorithm significantly, while the value of $p$ serves to simplify the simulation of the forward problem used in the tests.
In addition, the linear case is relevant for the currently known imaging modalities.

We consider the two physically motivated scenarios with $q \equiv 1$ and $q \equiv 2$.
The value~$p-q$ determines whether there is an explicit formula for the solutions and which of the three scenarios is the case, but otherwise does not significantly alter the algorithms.
Note that when implementing the following algorithms, due to floating point precision, equality between quantities should be interpreted as observing an absolute difference smaller than a predetermined precision level. 
We consider domains with closed boundaries, whence a cyclic ordering of the boundary points is implicit.
We also assume the existence of an interpolation algorithm represented by $\mathtt{interpolate}$. 

A matter of notation: We use $:=$ to denote assignment of values, and for ease of reading we write $A_j, N_j, H_j$ for the values $A(x_j), N(x_j), H(x_j)$ sampled at $M$ boundary points $x_j, j = 1,...,M$.

We begin with the simpler of the two problems, $d = p = 2$ and $q = 1$, corresponding to the CDII modality treated in corollary \ref{cor:p2q1}. 
In this modality, the algorithm consists of directly computing the estimates. 
It makes use of a priori upper and lower bounds on $\sigma$ as well as an index set $U$ containing the indices of undecided points.

\begin{algorithm}[Parameters: d = p = 2, q = 1]
\begin{algorithmic}
\label{alg:easy_alg}
\STATE{}
\STATE{\textbf{Input:} Bounds $\underline \sigma, \overline \sigma$ and samples $x_j, A_j, N_j, H_j, j = 1,...,M$. }
\FOR{$j = 1,...,M$}
\IF{$A_j \neq 0 \text{ and } \underline \sigma \leq  \real \left\{ \sqrt{H_j^2 - N_j^2 } \right\}/A_j \leq \overline \sigma$ }
\STATE{$\sigma_{\text{est}}(x_j) := \real \left\{ \sqrt{H_j^2 - N_j^2 } \right\}/A_j$}
\ELSE
\STATE{$U := U \cup \{ j \}$}
\ENDIF
\ENDFOR
\FOR{$j \in U$}
\STATE{$\sigma_{\text{est}}(x_j) = \mathtt{interpolate}\left( \left\lbrace \sigma_{\text{est}}(x_j)| j \in U^C\right\rbrace \right)$}
\ENDFOR
\end{algorithmic}
\end{algorithm}

In the next modality, AET/UMEIT, as discussed in corollary  \ref{cor:pq_2}, one must take care to select the correct candidate whenever possible, as outlined in section~\ref{sec:select}. 
We use the term \textit{double candidate} for points~$x_j$ where corollary~\ref{cor:pq_2} predicts a double root.
A double candidate is labelled \textit{undecided} if both $\sigma^+$ and $\sigma^-$ are admissible solutions or if neither is admissible.
Otherwise, we label it \textit{decided}.
Double candidates with two valid values (between lower and upper a~priori limits on $\sigma$) are undecideable on their own, yet it may be possible to use information from neighbouring points to pick a candidate.

If the double candidate $x_j$ is undecided, but its neighbours $x_{j-1}$ and $x_{j+1}$ agree on the use of either $\sigma^+$ or $\sigma^-$, it is reasonable to choose the value of $\sigma$ at $x_j$ accordingly.
We extend this logic to finite sequences of consecutive double candidates.
By searching in both directions from $x_j$ to find points $x_k$ and $x_l$ where we have made a choice of $\sigma^+$ or $\sigma^-$, we can decide whether to use $\sigma^+$ or $\sigma^-$ at all points between $x_k$ and $x_l$.

It is possible that the search encounters a point indicating a loss of information about which candidate to pick.
We call these points \textit{stopping points}, and terminate the search when encountering them.
The first type of stopping points $x_m$ consists of those where $H_m = 0$, i.e.\ it is impossible to compute a candidate for~$\sigma$. 
The second type of stopping point is encountered when passing points where $n^+ = n^-$.
Since $n^+$ and $n^-$ have the same sign, equality occurs when  $ \Delta n = \abs{n_+} - \abs{n_-}  = 0$.
By lemma~\ref{lem:n_difference}, $ \Delta n = \abs{n_+} - \abs{n_-}  \ge 0$.
Thus, $\Delta n$ is minimized at any point $x_m$ where $n^+ = n^-$, yet the converse is not true.
We must take into account that we have a finite amount of sampling points and so will probably miss the exact minimizer. 
We therefore consider a point $x_m$ which is a local minimum for the sequence $\{\Delta n(x_j)\}_{j=1}^M$ as a possible stopping point.
If, in addition, $\Delta n(x_m)$ is smaller than some predetermined threshold~$\epsilon$, it is considered a stopping point.

If the search ends in a stopping point $x_m$ in one direction and a decided point $x_d$ in the other, we set all choices of $\sigma^+$ or $\sigma^-$ at all points between $x_m$ and $x_d$ in accordance with $x_d$.

We use three index sets, $D, S$ and $U$, to label point indices as double candidates, stopping points and/or undecided, respectively.
\begin{algorithm}[Parameters: $d = p = q = 2$]
\begin{algorithmic}
\label{alg:hard_alg}
\STATE{}
\STATE{\textbf{Input:} Bounds $\underline \sigma, \overline \sigma, \epsilon$, measurements $x_j, A_j, N_j, H_j, j = 1,...,M$. }
\STATE{\textbf{Initialize:} $D := \emptyset, S := \emptyset, U := \emptyset$}
\FOR{$j = 1,...,M$}
\IF{$H_j= 0$ }
\STATE{$U := U \cup \{ j \}$}
\STATE{$S := S \cup \{ j \}$}
\ELSIF{$A_j \neq 0$ and $N_j = 0$ and $ \underline \sigma \leq H_j/A^2_j \leq \overline \sigma$}
\STATE{$\sigma_{\text{est}}(x_j) := H_j/A^2_j$}
\ELSIF{$A_j = 0$ and $N_j \neq 0$ and $ \underline \sigma \leq N^2_j/H_j \leq \overline \sigma$}
\STATE{$\sigma_{\text{est}}(x_j) := N^2_j/H_j$}
\ELSE
\STATE{$D := D \cup \{ j \}$}
\STATE{$\sigma^+_{\text{est}}(x_j) := \mathrm{Re}\left\lbrace2N_j^2/\left(H_j + \sqrt{(H_j^2 - 4A_j^2)N_j^2}\right)\right\rbrace$}
\STATE{$\sigma^-_{\text{est}}(x_j) := \mathrm{Re}\left\lbrace2N_j^2/\left(H_j - \sqrt{(H_j^2 - 4A_j^2)N_j^2}\right)\right\rbrace$}
\IF{$\sigma^+_{\text{est}}(x_j) = \sigma^-_{\text{est}}(x_j) \, \mathrm{ and } \, \underline \sigma \leq \sigma^+_{\text{est}}(x_j) \leq \overline \sigma$ }
\STATE{$\sigma_{\text{est}}(x_j) := \sigma^+_{\text{est}}(x_j)$}
\ELSIF{$(\sigma^-_{\text{est}}(x_j) < \underline \sigma \, \mathrm{ or } \, \sigma^-_{\text{est}}(x_j) > \overline \sigma) \, \mathrm{ and } \, \underline \sigma \leq \sigma^+_{\text{est}}(x_j) \leq \overline \sigma$}
\STATE{$\sigma_{\text{est}}(x_j) := \sigma^+_{\text{est}}(x_j)$}
\ELSIF{$(\sigma^+_{\text{est}}(x_j) < \underline \sigma \, \mathrm{ or } \, \sigma^+_{\text{est}}(x_j) > \overline \sigma) \, \mathrm{ and } \, \underline \sigma \leq \sigma^-_{\text{est}}(x_j) \leq \overline \sigma$}
\STATE{$\sigma_{\text{est}}(x_j) := \sigma^-_{\text{est}}(x_j)$}
\ELSE
\STATE{$U := U \cup \{ j \}$}
\ENDIF
\ENDIF
\IF{$\Delta n(x_j) < \Delta n(x_{j+1}) \text{ and } \Delta n(x_j) < \Delta n(x_{j-1}) \text{ and } |\Delta n(x_j)| < \epsilon$ }
\STATE{$S := S \cup \{ j \}$}
\ENDIF
\ENDFOR
\FOR{$j \in (D \cap U) \backslash S $}
\STATE{$k := \max\{i \in \mathbb{N} |i < j \text{ and } i \in S \text{ or } i \in D \cap U^C \}$}
\STATE{$l := \min\{i \in \mathbb{N} |i > j \text{ and } i \in S \text{ or } i \in D \cap U^C \}$}
\IF{$\left\lbrace
\begin{aligned}
&k \in U^C ,& l  \in U^C&;& \sigma_{\text{est}}(x_k) = \sigma^-_{\text{est}}(x_k) &,\, \sigma_{\text{est}}(x_l) = \sigma^-_{\text{est}}(x_l)\\
\text{or } &l \in U^C,& k \in U&;& \sigma_{\text{est}}(x_l) &= \sigma^-_{\text{est}}(x_l)\\
\text{or } &k \in U^C,& l \in U&;& \sigma_{\text{est}}(x_k) &= \sigma^-_{\text{est}}(x_k)
\end{aligned}
\right\rbrace$}
\STATE{$\sigma_{\text{est}}(x_j) := \sigma^-_{\text{est}}(x_j)$ for $j \in \{k+1,...,l-1 \}$}
\STATE{$U := U \backslash \{k+1,...,l-1 \}$}
\ELSIF{$\left\lbrace
\begin{aligned}
&k \in U^C ,&l  \in U^C&;& \sigma_{\text{est}}(x_k) = \sigma^+_{\text{est}}(x_k)&,\, \sigma_{\text{est}}(x_l) = \sigma^+_{\text{est}}(x_l)\\
\text{or } &l \in U^C,&k \in U&;& \sigma_{\text{est}}(x_l) &= \sigma^+_{\text{est}}(x_l)\\
\text{or } &k \in U^C,&l \in U&;& \sigma_{\text{est}}(x_k) &= \sigma^+_{\text{est}}(x_k)
\end{aligned}
\right\rbrace$}
\STATE{$\sigma_{\text{est}}(x_j) := \sigma^+_{\text{est}}(x_j)$ for $j \in \{k+1,...,l-1 \}$}
\STATE{$U := U \backslash \{k+1,...,l-1 \}$}
\ENDIF
\ENDFOR
\FOR{$j \in U$}
\STATE{$\sigma_{\text{est}}(x_j) = \mathtt{interpolate}\left( \left\lbrace \sigma_{\text{est}}(x_j)| j \in U^C\right\rbrace \right)$}
\ENDFOR
\end{algorithmic}
\end{algorithm}

\begin{remark}[Multiple measurements]
We have only considered the situation of a single measurement.
It is not immediately obvious how to combine several such measurements in a principled way.
We would recommend first doing the reconstruction for all points where it can be done uniquely, taking averages of different reconstructions when several are available, and then checking if some candidates for the conductivity are approximately equal in all reconstructions.
The reconstructions of the true conductivities should be roughly equal, whereas the false candidates need not be.
\end{remark}

\subsection{Numerical experiments}
We demonstrate algorithm~\ref{alg:hard_alg} by reconstructing conductivity on the boundary from numerically simulated AET data.
We consider algorithm~\ref{alg:hard_alg} because it should contain most of the difficulties that algorithm~\ref{alg:easy_alg} contains.

We implemented the code in MATLAB and executed it using MATLAB (2018b release) running on a mid-2014 MacBook Pro.
The code is available on GitHub\footnote{\url{https://github.com/tringholm/bdry-data-calderon}}.
We synthesized test measurement data for $A(x)$, $N(x)$ and $H(x)$ by solving the linear conductivity equation, equation \eqref{eq:linear_cond_eqn}, with given conductivity and boundary data.
To do so, we used the finite element solver with quadratic elements from MATLAB's Partial Differential Equation toolbox.
We used a triangulated unit disk as the domain~$\Omega$.
The inverse crime~\cite[chapter~2]{Mueller:Siltanen:2012} was avoided by sampling boundary at points independent of the finite element grid.
We sampled uniformly at $M$ points along the unit circle in a counter-clockwise fashion.

We added noise to several measured quantities to assess the robustness of the algorithm.
\begin{itemize}
\item Uncertainty in measurement location was simulated using additive Gaussian noise with variance $2\pi/M \cdot 5\%$ in the angular position of the measurement points, i.e. measuring at slightly imprecise angles.
\item The Neumann measurements $N_i$ were subjected to additive Gaussian noise with variance $N_i \cdot 5\%$.
\item The current density measurements $H_i$ were subjected to additive Gaussian noise with variance $H_i \cdot 5\%$.
\end{itemize}

Figures \ref{fig:easy_figure_1000pts}, \ref{fig:easy_figure_100pts} and \ref{fig:hard_figure_100pts} illustrate the reconstruction steps of algorithm~\ref{alg:hard_alg}. 
First we determine non-double candidate points. Next, we add unambiguous double points, i.e. double candidates where one candidate breaks the prior bounds. 
After this, we choose a value for remaining undecided double points if possible by the inferment procedure described before algorithm~\ref{alg:hard_alg}. 
Then, undecideable points are interpolated using linear interpolation. 
Finally, we post-process the reconstructed data by applying a Gaussian smoothing and compare with the exact solution.
In figures \ref{fig:easy_figure_1000pts} and \ref{fig:easy_figure_100pts}, a conductivity of 
\begin{align}
\sigma(x) = \dfrac{3}{1 + \mathrm{e}^{2(x_1+x_2)}}
\end{align}
was used, with reconstruction using $M = 1000$ and $M = 100$ samples, respectively.
The last figure shows a test of the reconstruction procedure using $M = 100$ samples with a more oscillatory conductivity,
\begin{align}
\sigma(x) = 2 + \cos(10(x_1-x_2)).
\end{align}
In both cases, the boundary data was chosen as 
\begin{align*}
u|_{\partial \Omega}(x) = \max(0,x_1).
\end{align*}
to include a part of the boundary with $A = 0$ such that we could observe the different cases present in the algorithm.

\newpage
\begin{figure}
\includegraphics[width=\figscale\textwidth]{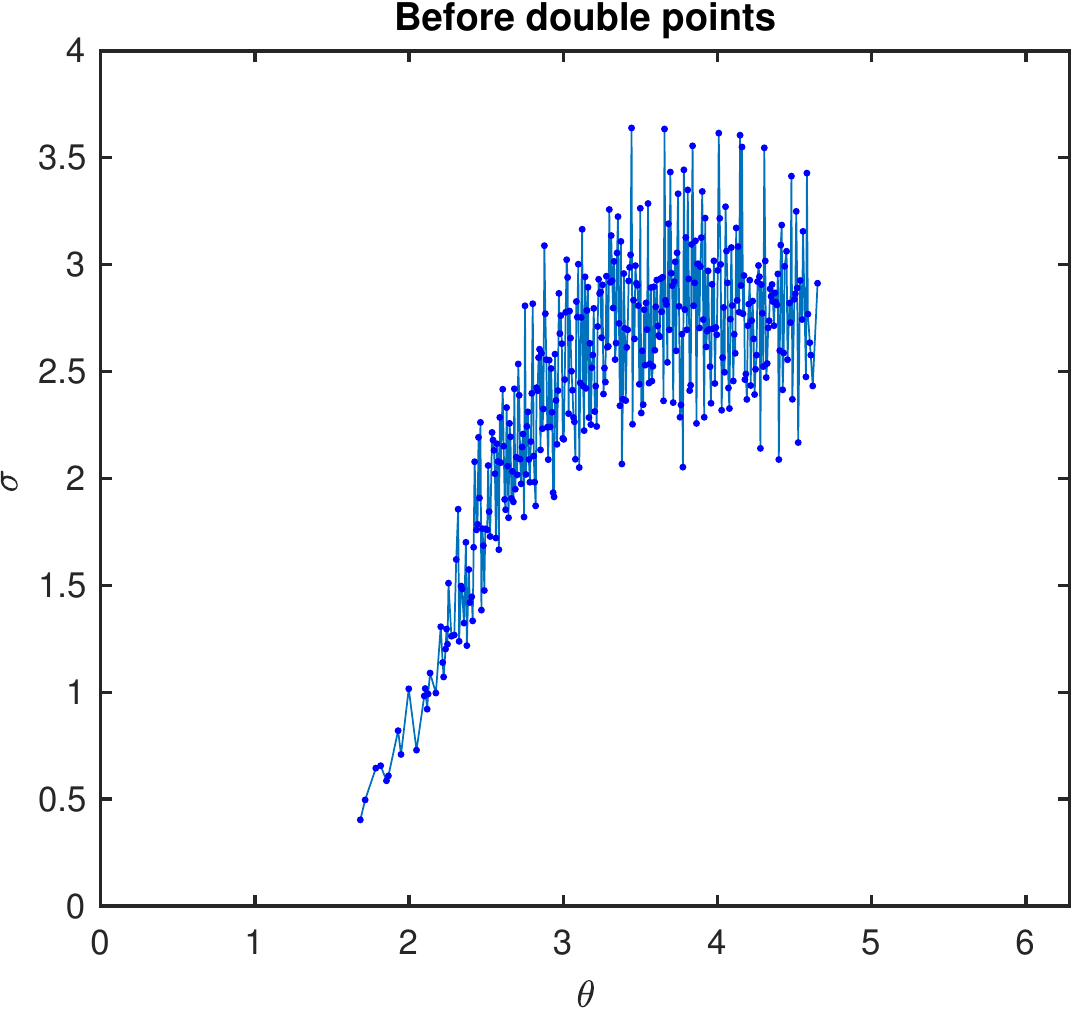}
\includegraphics[width=\figscale\textwidth]{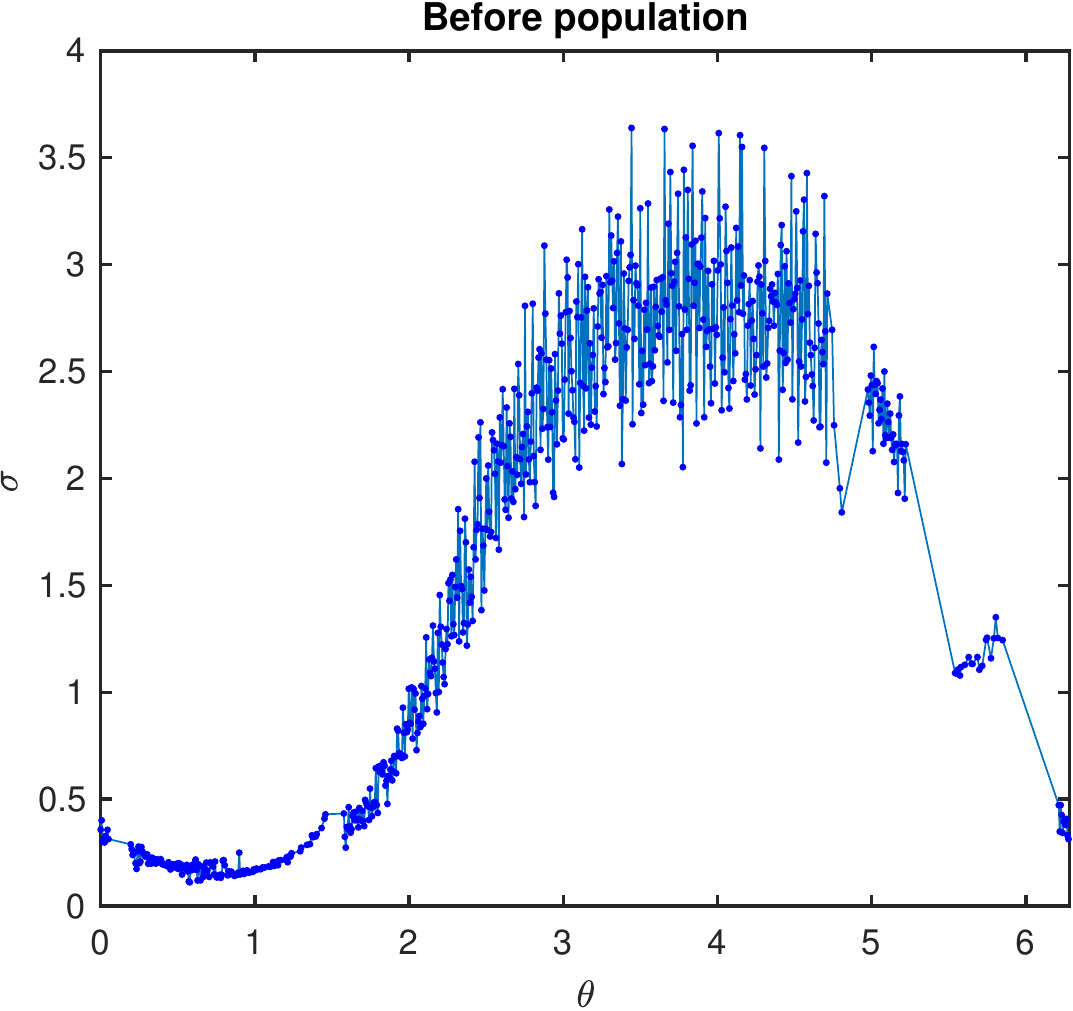}\\
\includegraphics[width=\figscale\textwidth]{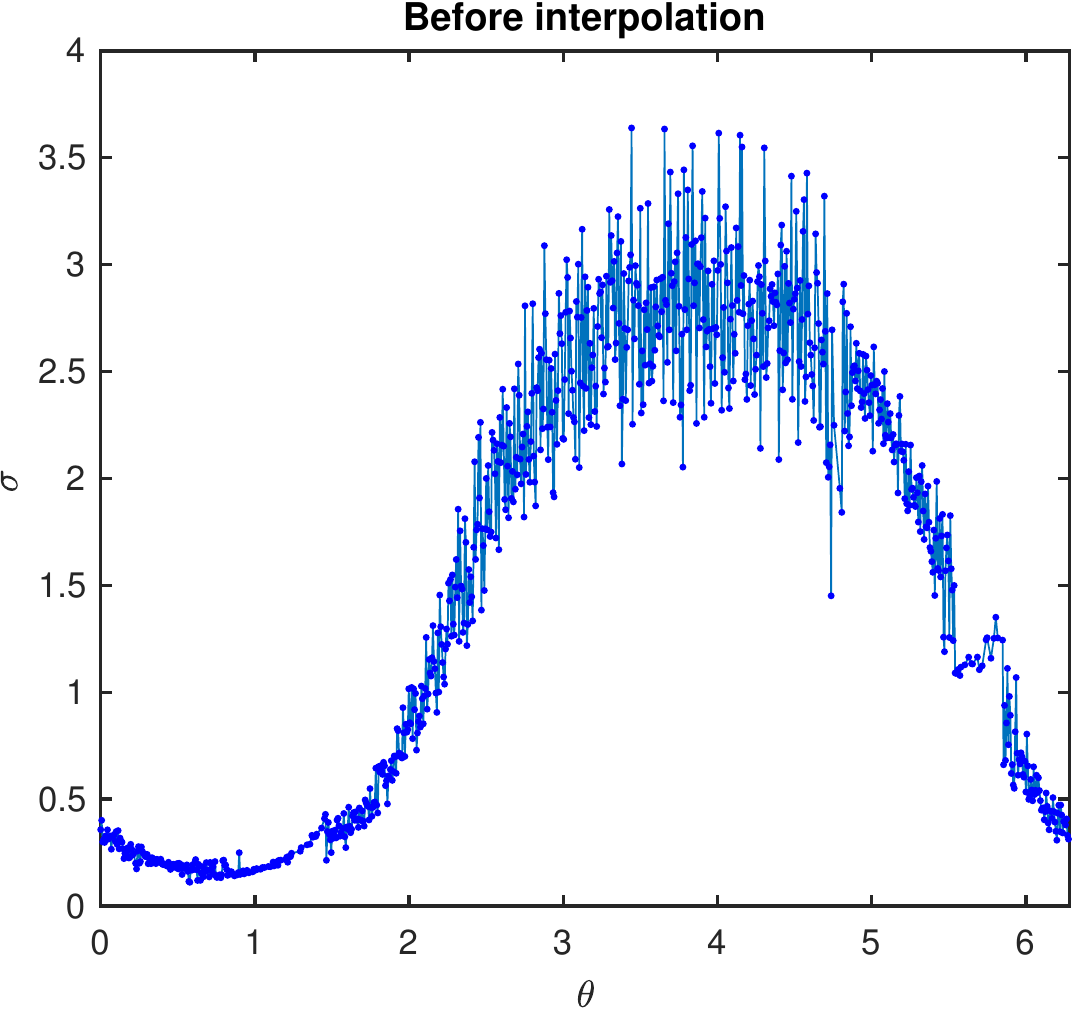}
\includegraphics[width=\figscale\textwidth]{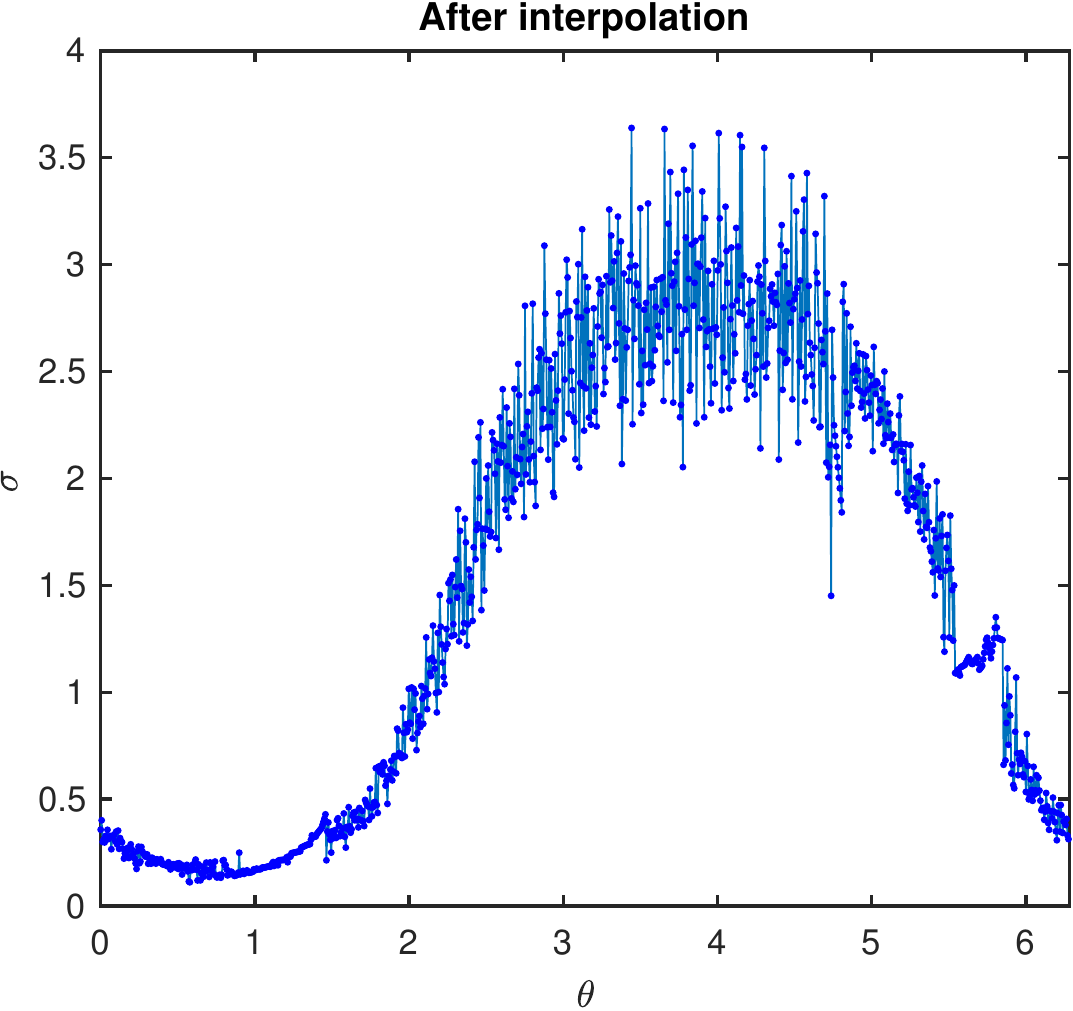}\\
\includegraphics[width=\figscale\textwidth]{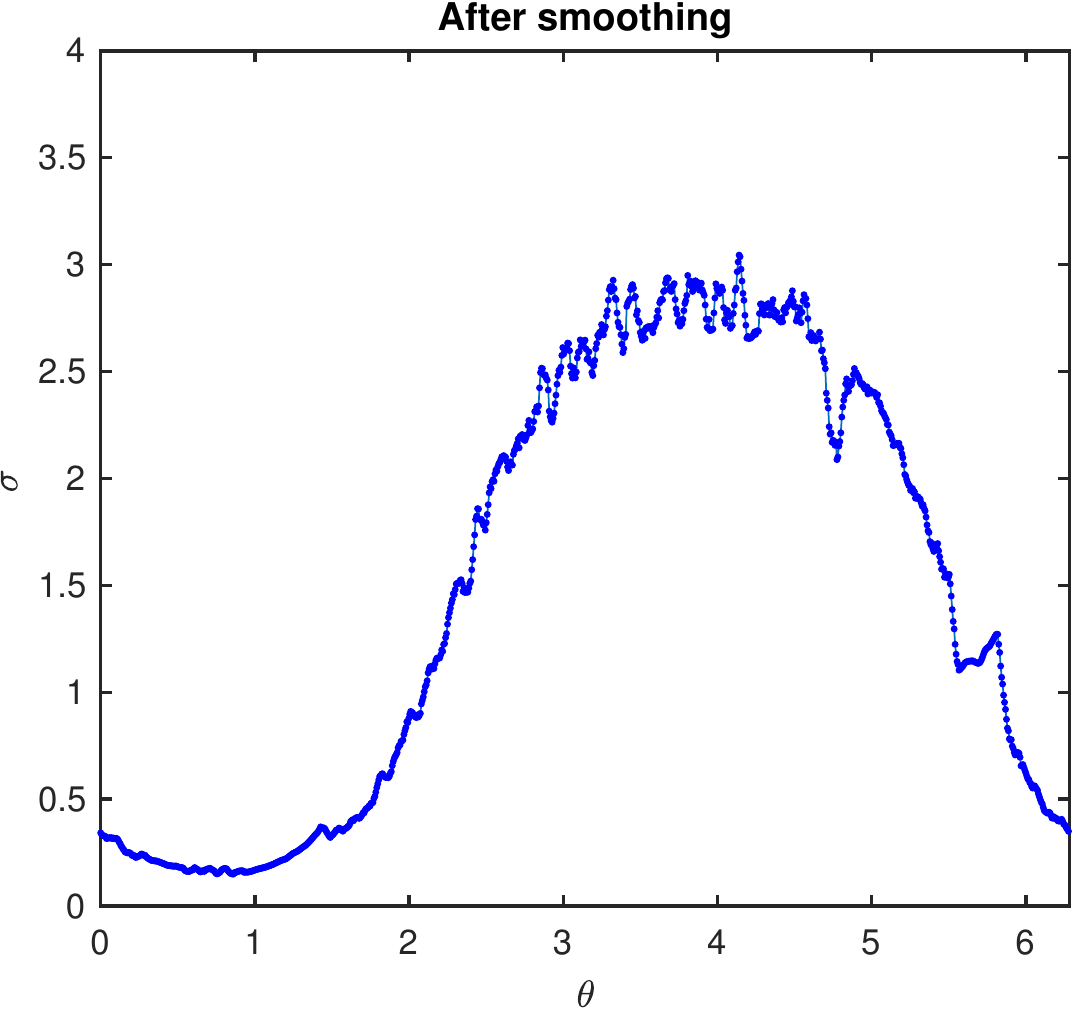}
\includegraphics[width=\figscale\textwidth]{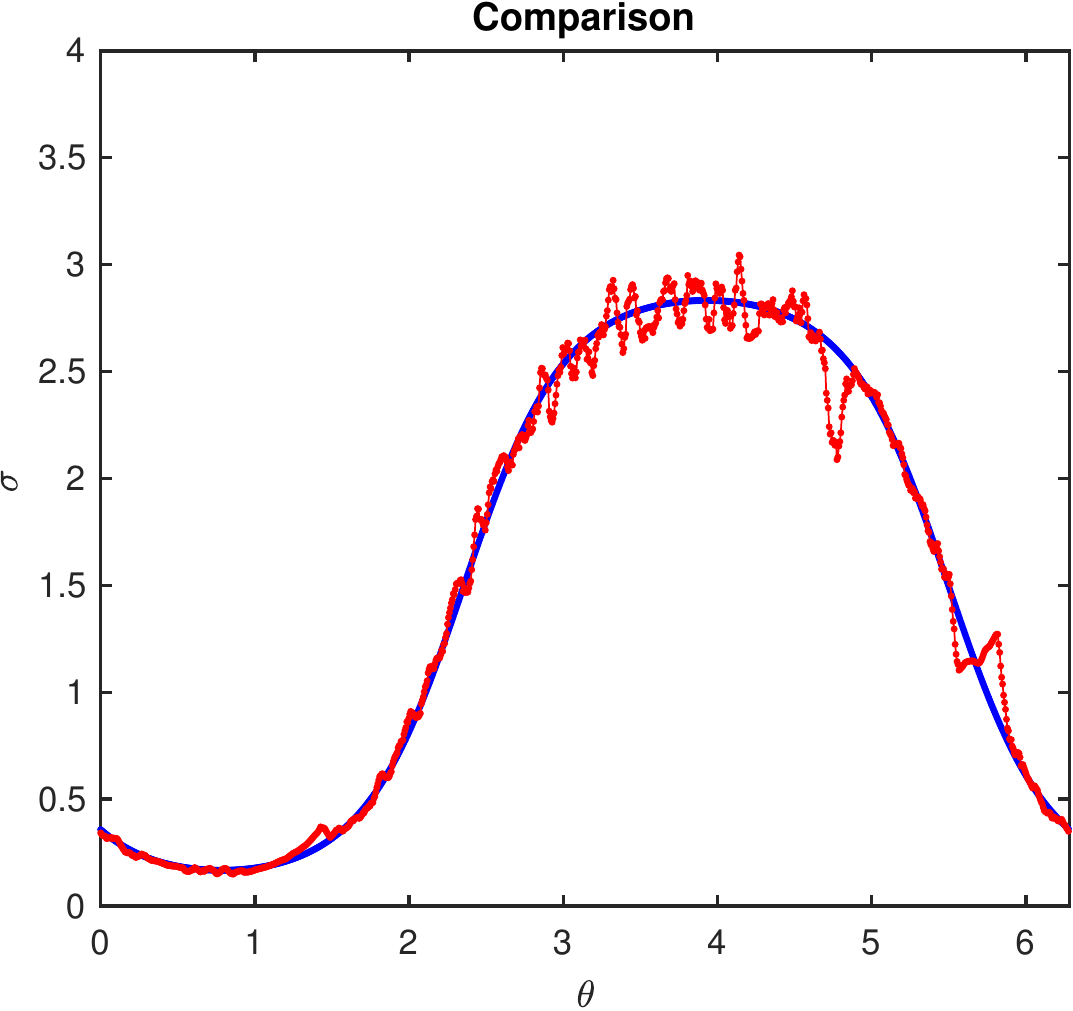}
\caption{Signal reconstruction using 1000 samples. }
\label{fig:easy_figure_1000pts}
\end{figure}

\begin{figure}
\includegraphics[width=\figscale\textwidth]{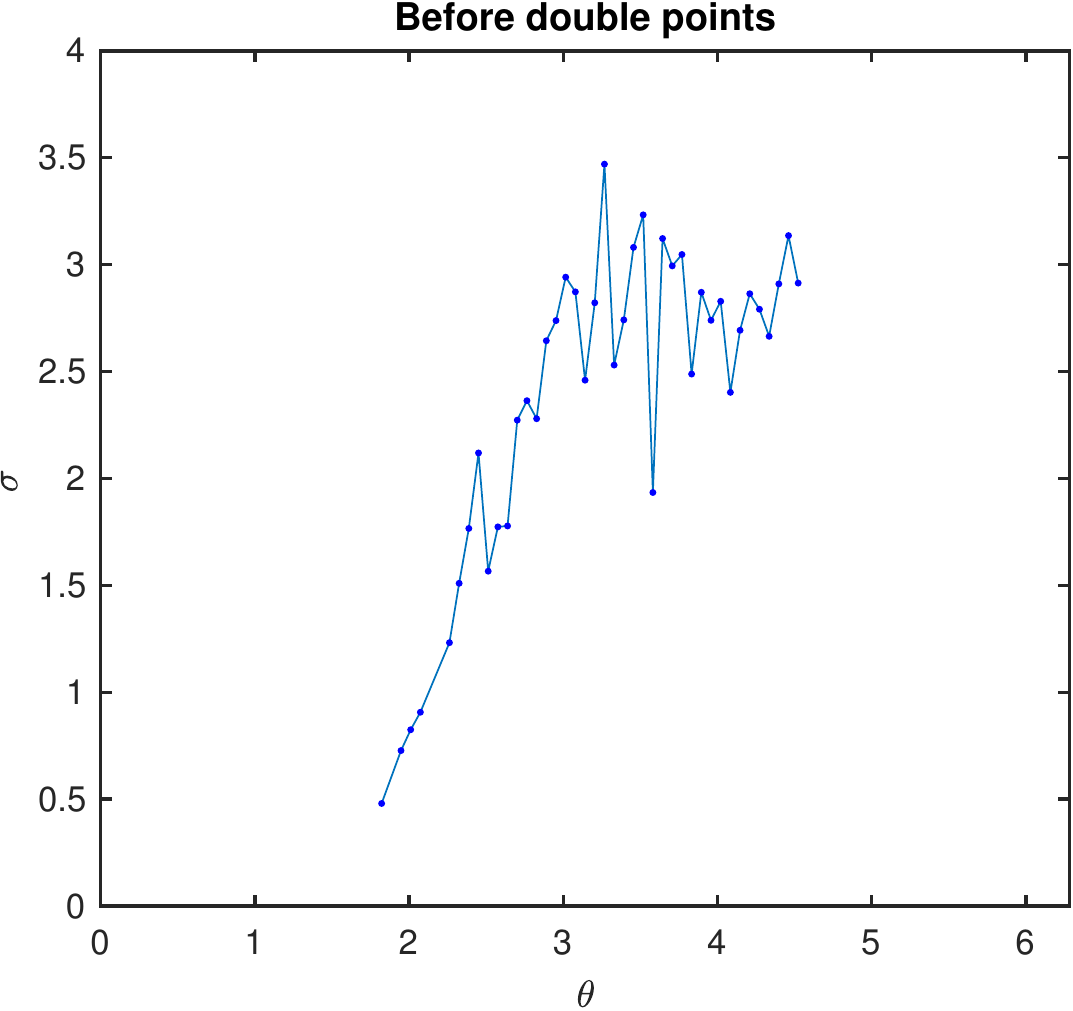}
\includegraphics[width=\figscale\textwidth]{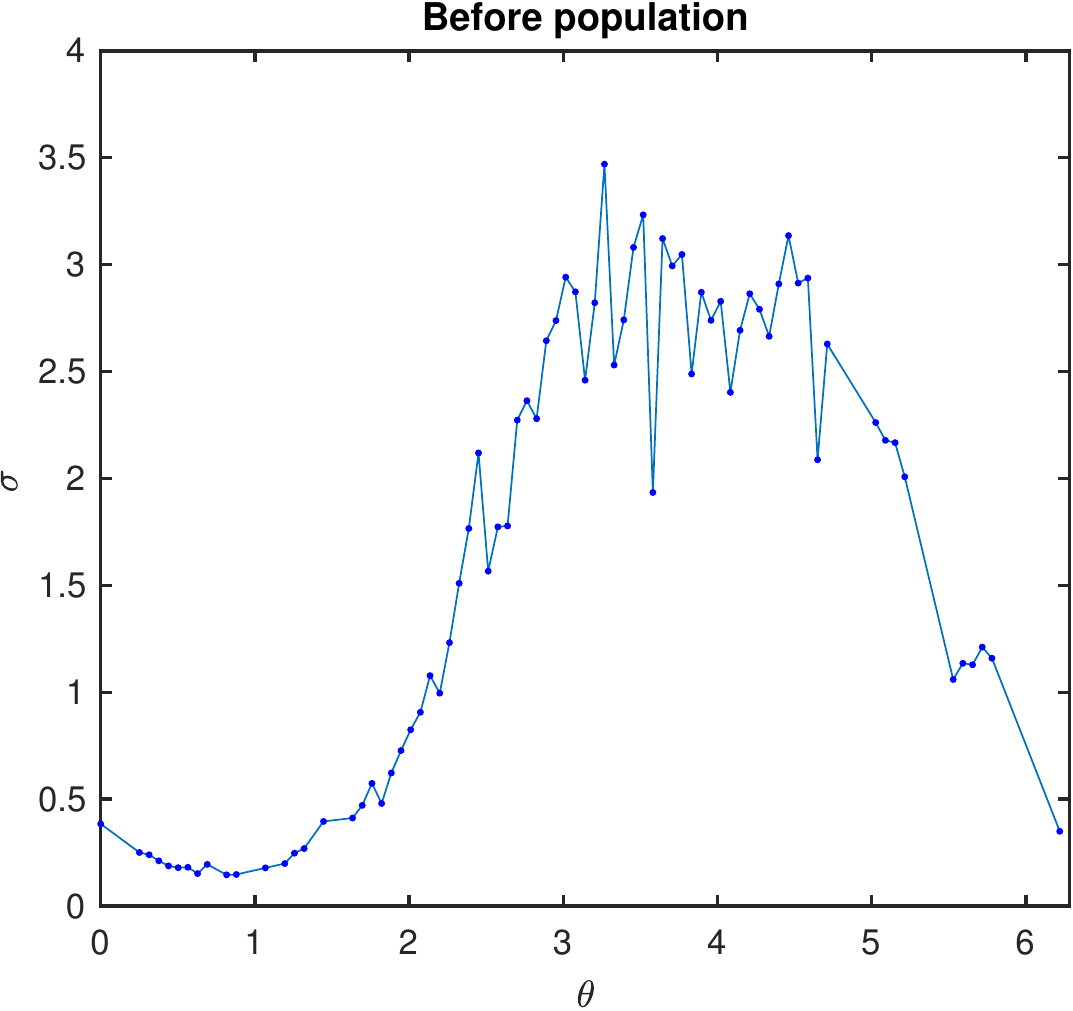}\\
\includegraphics[width=\figscale\textwidth]{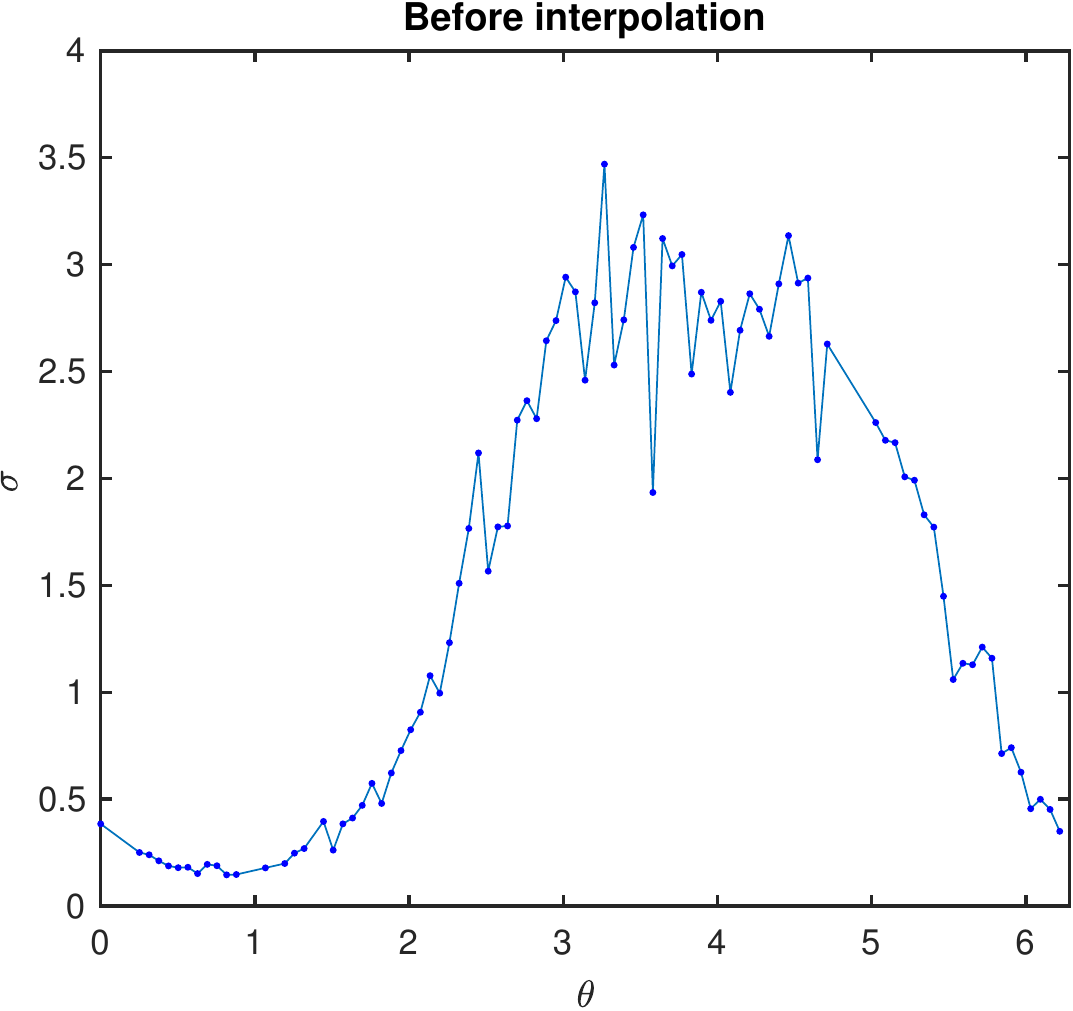}
\includegraphics[width=\figscale\textwidth]{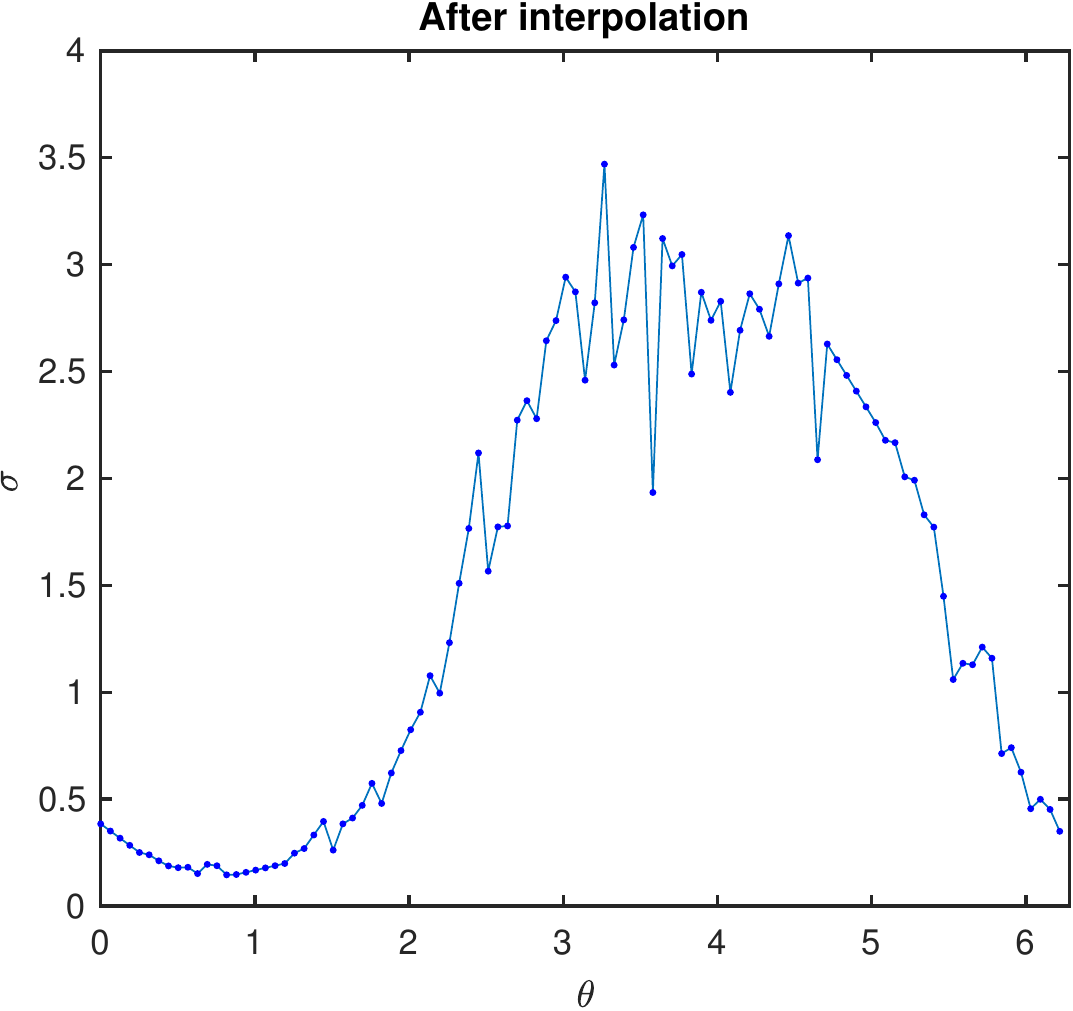}\\
\includegraphics[width=\figscale\textwidth]{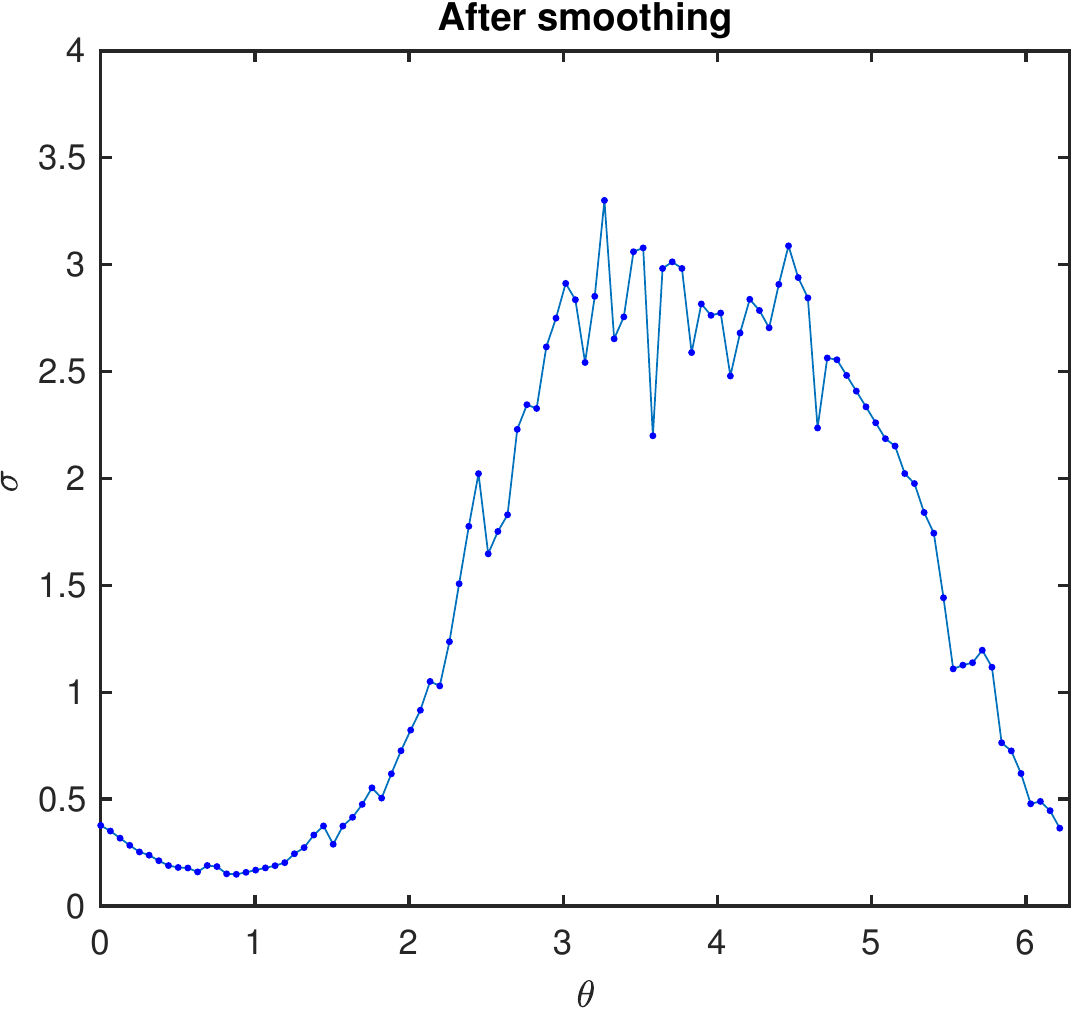}
\includegraphics[width=\figscale\textwidth]{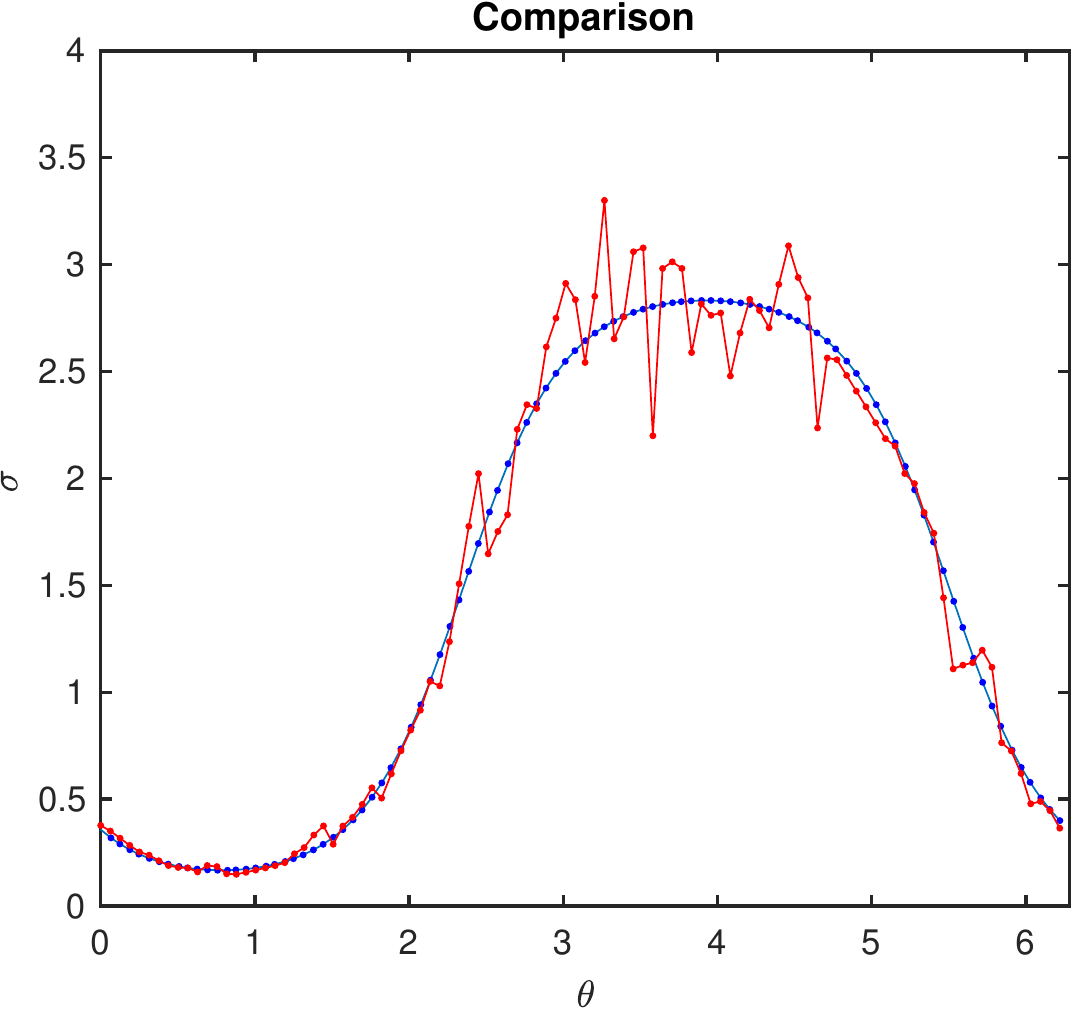}
\caption{Signal reconstruction using 100 samples.}
\label{fig:easy_figure_100pts}
\end{figure}


\begin{figure}
\includegraphics[width=\figscale\textwidth]{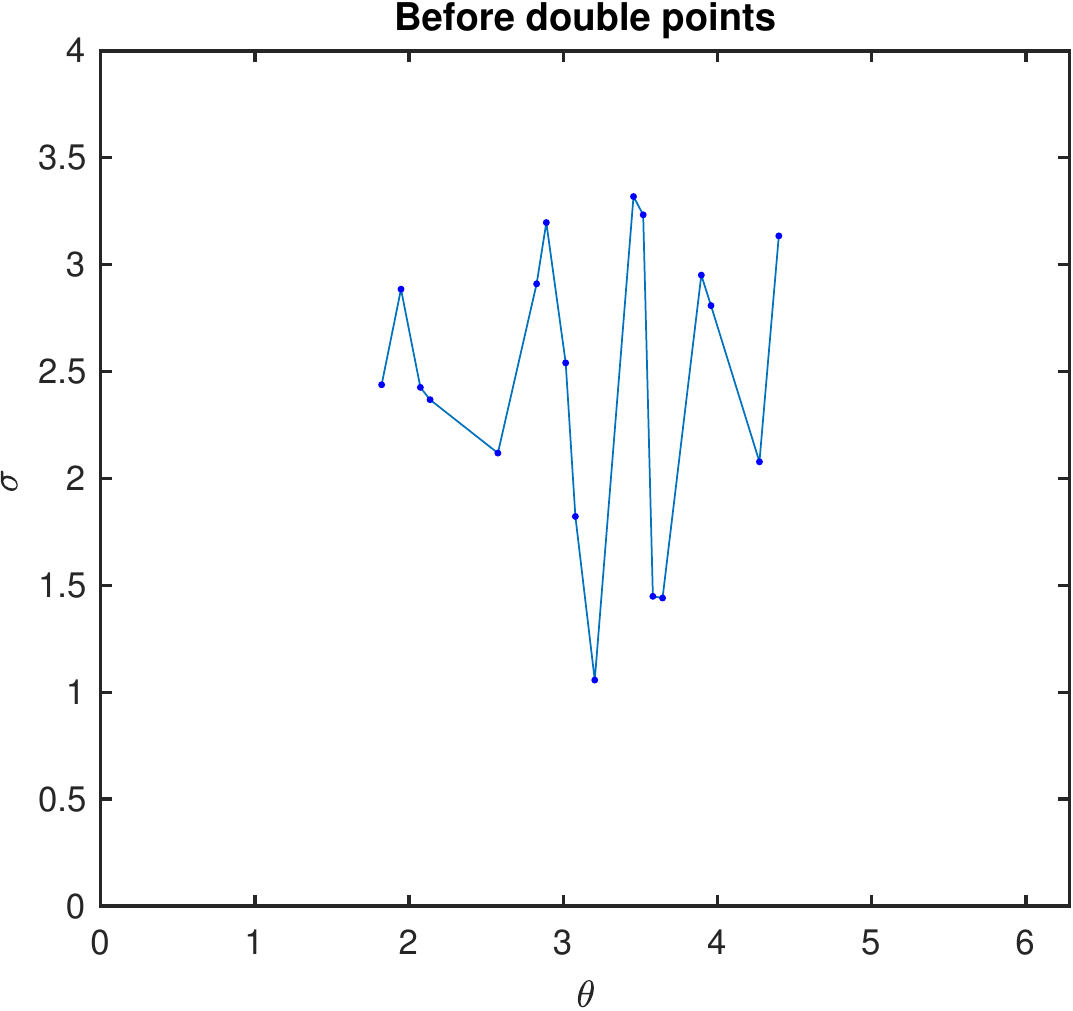}
\includegraphics[width=\figscale\textwidth]{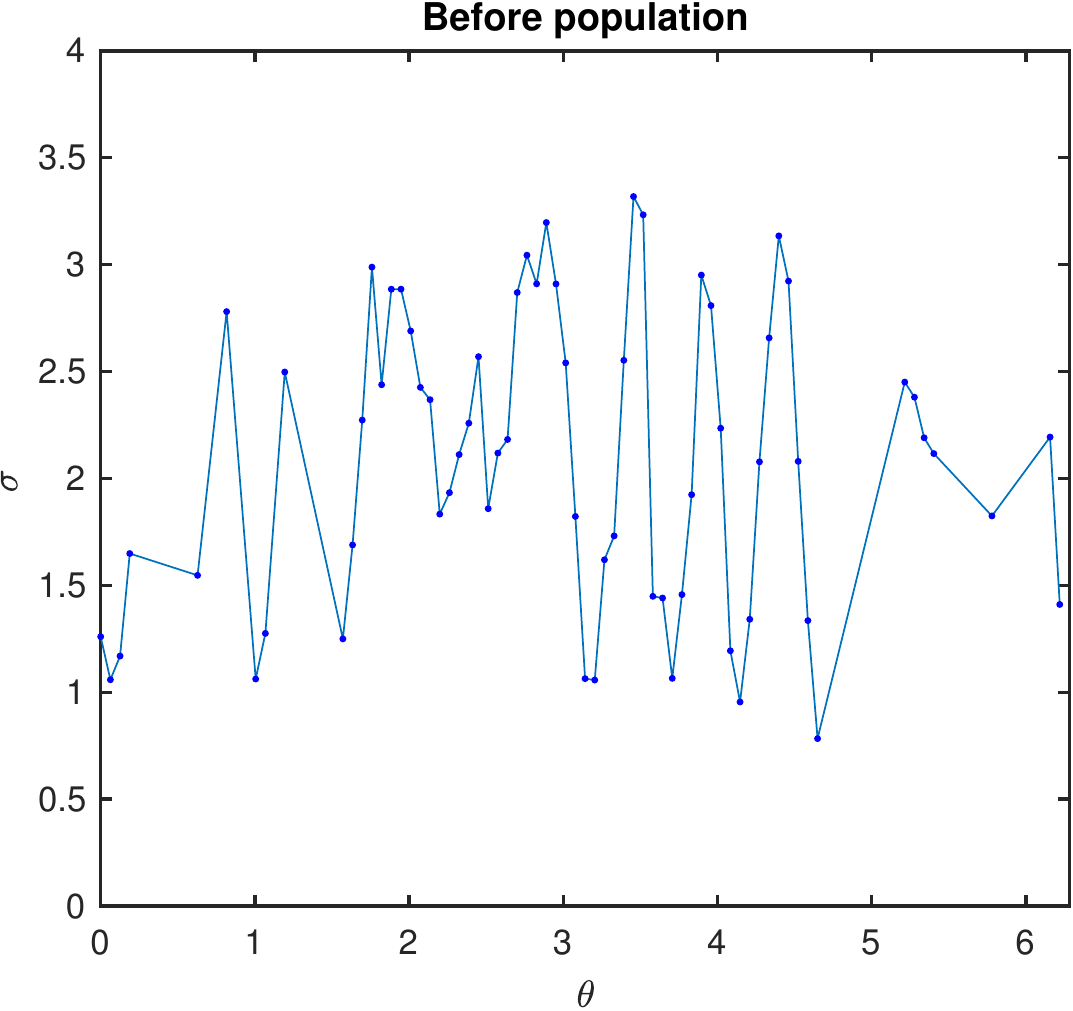}\\
\includegraphics[width=\figscale\textwidth]{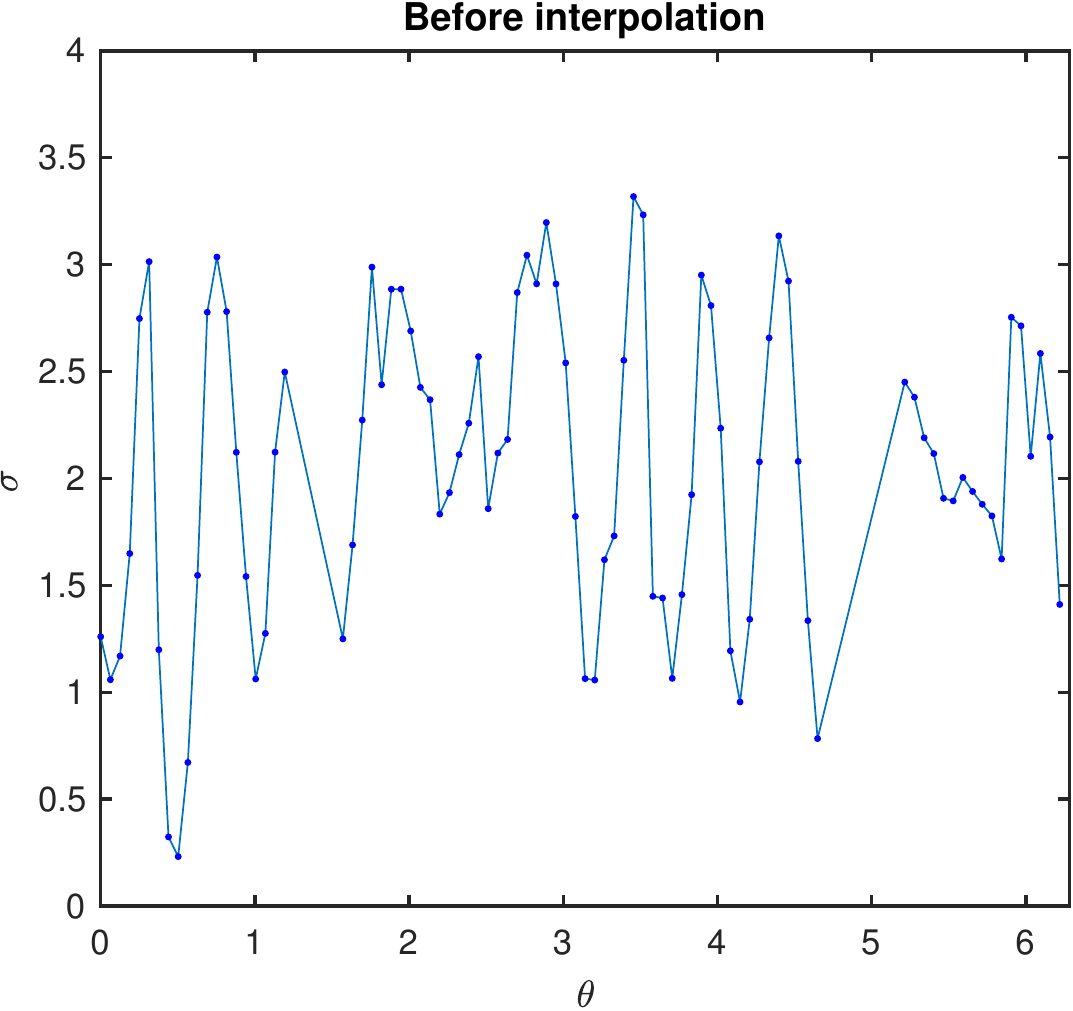}
\includegraphics[width=\figscale\textwidth]{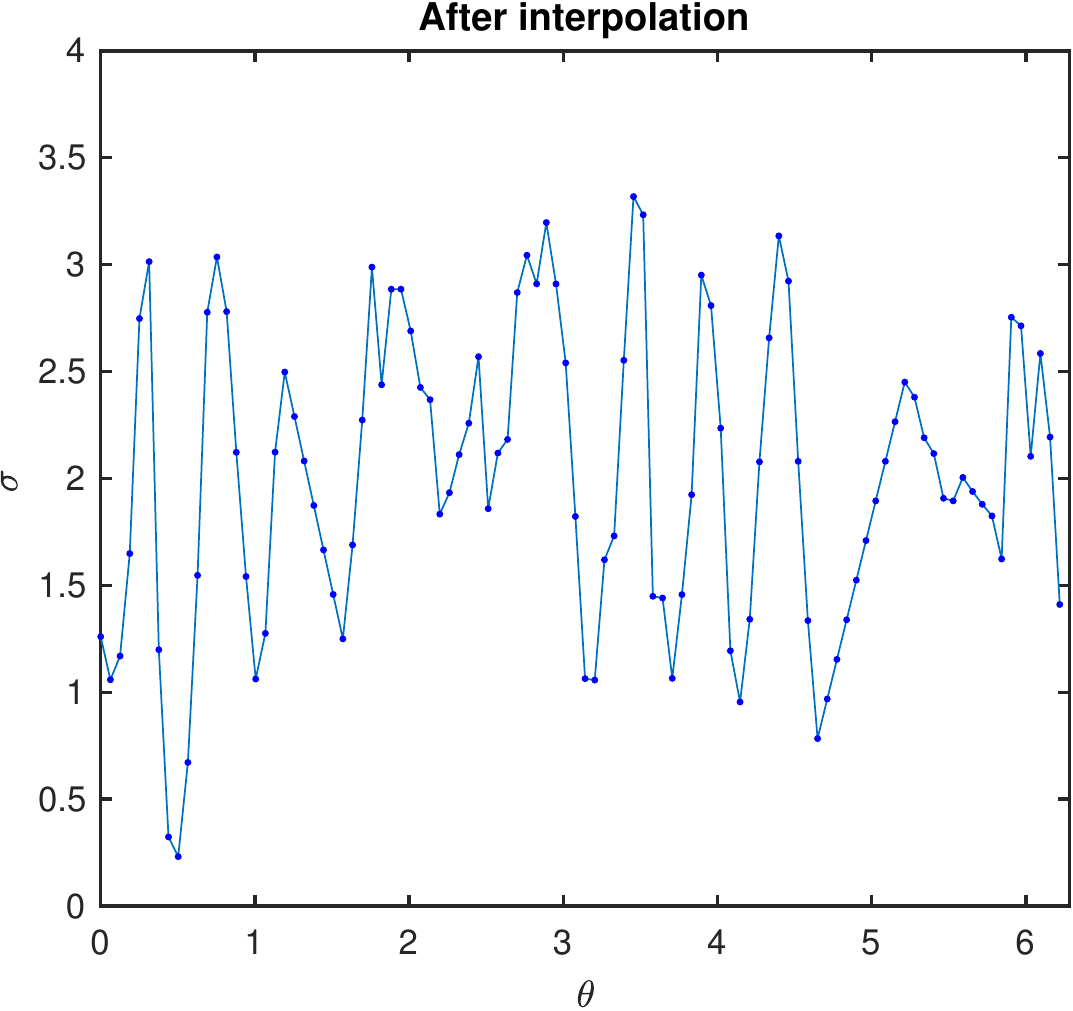}\\
\includegraphics[width=\figscale\textwidth]{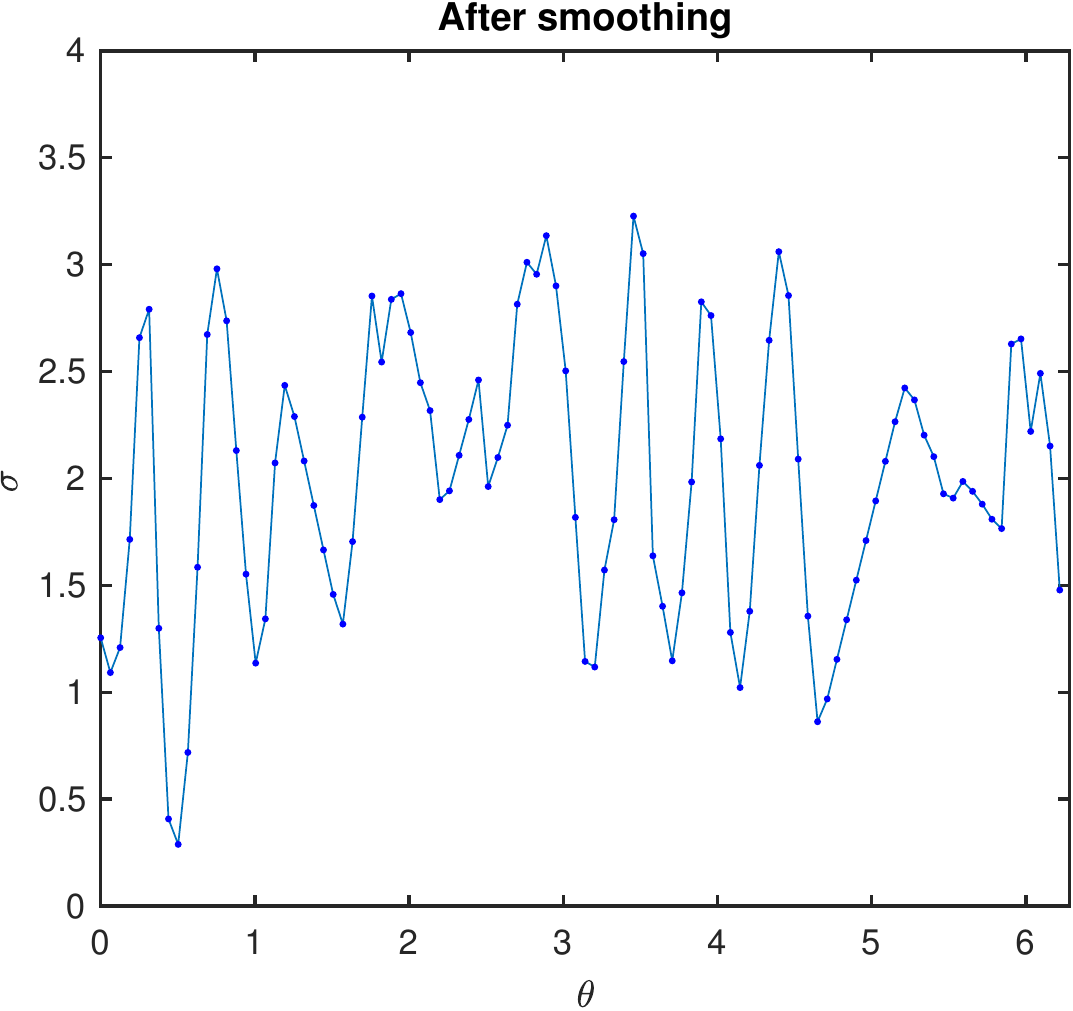}
\includegraphics[width=\figscale\textwidth]{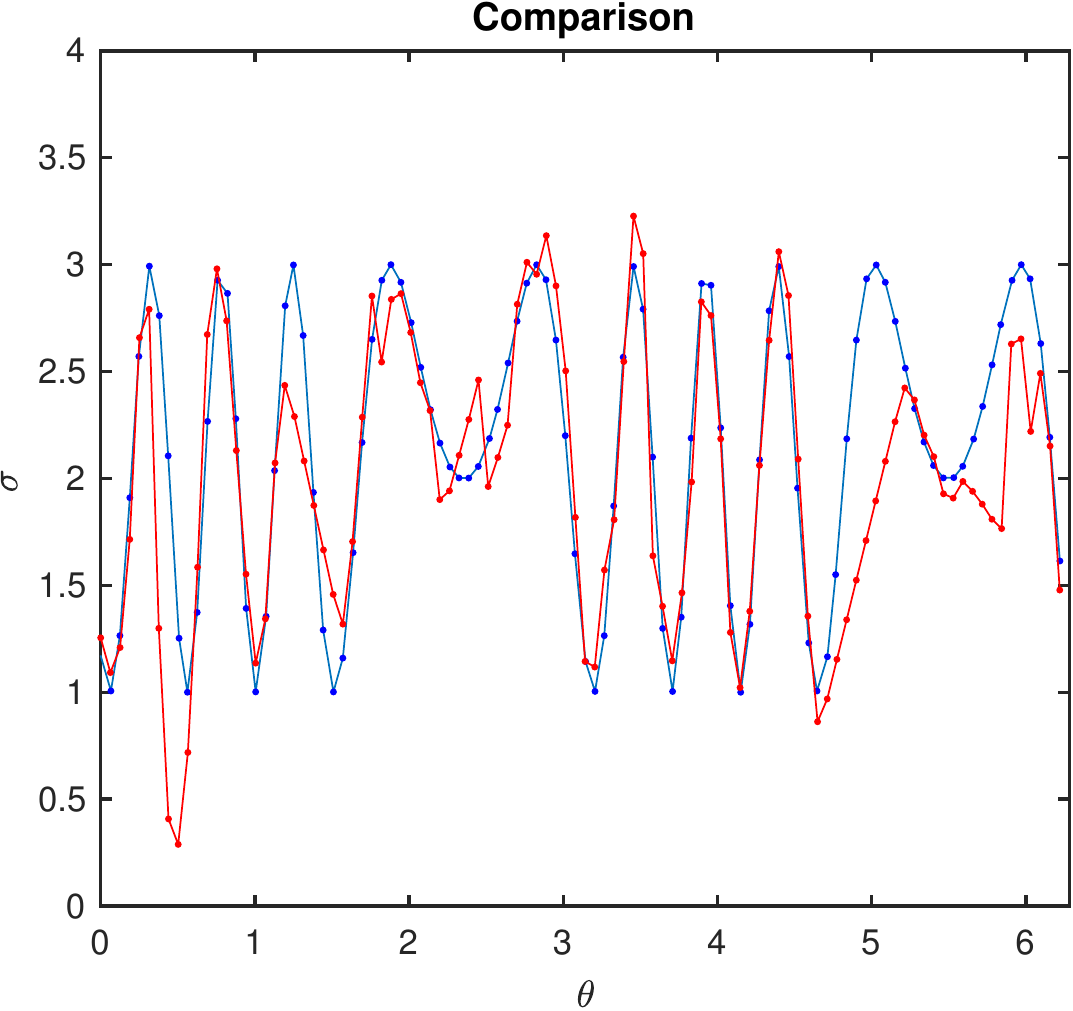}
\caption{Signal reconstruction using 100 samples, oscillatory conductivity.}
\label{fig:hard_figure_100pts}
\end{figure}

\clearpage

\bibliographystyle{plain}
\bibliography{math}

\begin{thebibliography}{10}

\bibitem{Alberti:Capdeboscq:2018}
Giovanni~S. Alberti and Yves Capdeboscq.
\newblock {\em Lectures on elliptic methods for hybrid inverse problems}.
\newblock Number~25 in Cours Sp{\'e}cialis{\'e}s. Soci{\'e}t{\'e}
  Math{\'e}matique de France, Paris, 2018.

\bibitem{Bal:2013:aug}
Guillaume Bal.
\newblock {Cauchy problem for Ultrasound Modulated EIT}.
\newblock {\em Analysis \& PDE}, 6(4):751--775, August 2013.

\bibitem{Bal:2013}
Guillaume Bal.
\newblock Hybrid inverse problems and internal functionals.
\newblock In Gunther Uhlmann, editor, {\em Inverse problems and applications:
  inside out. II}, volume~60 of {\em Mathematical sciences research institute
  publications}, pages 325--368. Cambridge university press, 2013.

\bibitem{Bal:Hoffman:Knudsen:2017}
Guillaume Bal, Kristoffer Hoffmann, and Kim Knudsen.
\newblock Propagation of singularities for linearised hybrid data impedance
  tomography.
\newblock {\em Inverse Problems}, 34(2):024001, December 2017.

\bibitem{Brander:2016:jan}
Tommi Brander.
\newblock Calder\'on problem for the $p$-{L}aplacian: First order derivative of
  conductivity on the boundary.
\newblock {\em Proceedings of American mathematical society}, 144:177--189,
  January 2016.
\newblock Preprint \href{http://arxiv.org/abs/1403.0428}{ar{X}iv:1403.0428}.

\bibitem{Brander:2016:apr}
Tommi Brander.
\newblock {\em Calder\'on's problem for $p$-{L}aplace type equations}.
\newblock PhD thesis, University of Jyv\"askyl\"a, Department of Mathematics
  and Statistics, Jyv\"askyl\"a, Finland, April 2016.
\newblock \url{http://urn.fi/URN:ISBN:978-951-39-6576-1}.

\bibitem{Brander:Harrach:Kar:Salo:2018}
Tommi Brander, Bastian von Harrach, Manas Kar, and Mikko Salo.
\newblock Monotonicity and enclosure methods for the $p$-{L}aplace equation.
\newblock {\em SIAM journal of applied mathematics}, 78(2):742--758, March
  2018.
\newblock Preprint \href{https://arxiv.org/abs/1703.02814}{ar{X}iv:1703.02814}.

\bibitem{Brander:Winterrose}
Tommi Brander and David Winterrose.
\newblock Variable exponent {C}alder{\'o}n's problem in one dimension.
\newblock {\em Annales Academiæ Scientiarum Fennicæ, Mathematica}, accepted.
\newblock Preprint \href{http://arxiv.org/abs/1808.04168}{ar{X}iv:1808.04168}.

\bibitem{Bueno:Longo:Varela:2008}
Paulo~R. Bueno, Jos\'e~A. Varela, and Elson Longo.
\newblock {SnO}$_2$, {ZnO} and related polycrystalline compound semiconductors:
  An overview and review on the voltage-dependent resistance (non-ohmic)
  feature.
\newblock {\em Journal of the European Ceramic Society}, 28(3):505--529, 2008.

\bibitem{Calderon:1980}
Alberto~Pedro Calder{\'o}n.
\newblock On an inverse boundary value problem.
\newblock In W.H. Meyer and M.A. Raupp, editors, {\em Seminar on numerical
  analysis and its applications to continuum physics}, pages 65--73. Sociedade
  {Brasileira} de Matematica, 1980.
\newblock Reprinted as \cite{Calderon:2006}.

\bibitem{Calderon:2006}
Alberto~Pedro Calder{\'o}n.
\newblock On an inverse boundary problem.
\newblock {\em Computation and applied mathematics}, 25(2--3):133--138, 2006.
\newblock Reprint of \cite{Calderon:1980}.

\bibitem{Capdeboscq:Fehrenbach:deGournay:Kavian:2009}
Yves Capdeboscq, J{\'e}r{\^o}me Fehrenbach, Fr{\'e}d{\'e}ric De~Gournay, and
  Otared Kavian.
\newblock Imaging by modification: numerical reconstruction of local
  conductivities from corresponding power density measurements.
\newblock {\em SIAM Journal on Imaging Sciences}, 2(4):1003--1030, October
  2009.

\bibitem{Diening:Harjulehto:Hasto:Ruzicka:2011}
Lars Diening, Petteri Harjulehto, Peter H{\"a}st{\"o}, and Michael
  R\r{u}\v{z}i\v{c}ka.
\newblock {\em Lebesgue and Sobolev spaces with variable exponents}, volume
  2017 of {\em Lecture Notes in Mathematics}.
\newblock Springer, 2011.

\bibitem{Dubson:Herbert:Calabrese:Harris:Patton:Garland:1988}
M.~A. Dubson, S.~T. Herbert, J.~J. Calabrese, D.~C. Harris, B.~R. Patton, and
  J.~C. Garland.
\newblock Non-{O}hmic dissipative regime in the superconducting transition of
  polycrystalline
  {${\mathrm{Y}}_{1}$${\mathrm{Ba}}_{2}$${\mathrm{Cu}}_{3}$${\mathrm{O}}_{\mathrm{x}}$}.
\newblock {\em Phys. Rev. Lett.}, 60:1061--1064, March 1988.

\bibitem{Fan:2007}
Xiangling Fan.
\newblock Global {$C^{1,\alpha}$} regularity for variable exponent elliptic
  equations in divergence form.
\newblock {\em Journal of differential equations}, 235:397--417, January 2007.

\bibitem{John:1982}
Fritz John.
\newblock {\em Partial differential equations.}, volume~1 of {\em Applied
  Mathematical Sciences}.
\newblock Springer, New York, NY, 4 edition, 1982.

\bibitem{Kang:Seo:2000}
Hyenonbai Kang and Jin~Keun Seo.
\newblock Recent progress in the inverse conductivity problem with single
  measurement.
\newblock In Gen Nakamura, Saburou Saitoh, Jin~Keun Seo, and Masahiro Yamamoto,
  editors, {\em Inverse problems and related topics}, number 419 in CRC
  research notes in mathematical sciences, pages 69--80. Chapman \& Hall, 2000.

\bibitem{Kar:Wang}
Manas Kar and Jenn-Nan Wang.
\newblock Size estimates for the weighted $p$-laplace equation with one
  measurement.
\newblock 2018.
\newblock Preprint
  \url{http://www.math.ntu.edu.tw/~jnwang/pub/resources/papers/size0614.pdf}.

\bibitem{Kuchment:Steinhauer:2012}
Peter Kuchment and Dustin Steinhauer.
\newblock Stabilizing inverse problems by internal data.
\newblock {\em Inverse Problems}, 28(8):084007, 20, July 2012.

\bibitem{Kwon:Woo:Yoon:Seo:2002}
O.~Kwon, E.~J. Woo, J.-R. Yoon, and J.~K. Seo.
\newblock Magnetic resonance electrical impedance tomography ({MREIT}):
  simulation study of {$J$}-substitution algorithm.
\newblock {\em IEEE Transactions on Biomedical Engineering}, 49(2):160--167,
  February 2002.

\bibitem{Mueller:Siltanen:2012}
Jennifer~L. Mueller and Samuli Siltanen.
\newblock {\em Linear and nonlinear inverse problems with practical
  applications}, volume~10 of {\em Computational Science \& Engineering}.
\newblock Society for Industrial and Applied Mathematics (SIAM), Philadelphia,
  PA, 2012.

\bibitem{Nachman:Tamasan:Timonov:2011}
Adrian Nachman, Alexandru Tamasan, and Alexander Timonov.
\newblock Current density impedance imaging.
\newblock In Guillaume Bal, David Finch, Peter Kuchment, John Schotland, Plamen
  Stefanov, and Gunther Uhlmann, editors, {\em Tomography and inverse transport
  theory}, volume 559 of {\em Contemporary mathematics}, pages 135--150.
  American mathematical society, 2011.

\bibitem{Salo:Zhong:2012}
Mikko Salo and Xiao Zhong.
\newblock An inverse problem for the $p$-{L}aplacian: {B}oundary determination.
\newblock {\em {SIAM} J. Math. Anal.}, 44(4):2474--2495, March 2012.

\end{thebibliography}

\end{document}